\documentclass[10 pt]{amsart}
\usepackage[utf8]{inputenc} 
\usepackage[T1]{fontenc} 
\usepackage{amsmath}
\allowdisplaybreaks
\usepackage{amsxtra}
\usepackage{amscd}
\usepackage{amsthm}
\usepackage{amsfonts}
\usepackage{amssymb}
\usepackage{libertine}
\usepackage{scalerel}
\usepackage{tikz-cd}
\usepackage{eucal}
\usepackage[all]{xy}
\usepackage{graphicx}
\usepackage{relsize}
\usepackage{leftidx}
\usepackage{fullpage}
\usepackage[autostyle=true]{csquotes}
\usepackage{mathtools}
\usepackage{mathrsfs}
\usepackage[backend=biber, style=alphabetic, maxalphanames=99, natbib=true, maxnames=99]{biblatex}
\usepackage{hyperref}
\newtheoremstyle{newplain}
  {8pt}
  {8pt}
  {\itshape}  
  {}       
  {\rmfamily\scshape} 
  {.}         
  {5pt plus 1pt minus 1pt} 
  {}          

\theoremstyle{newplain}
\newtheorem*{theorem*}{Theorem}
\newtheorem{theorem}{Theorem}[section]
\newtheorem{definition}[theorem]{Definition}
\newtheorem{proposition}[theorem]{Proposition}
\newtheorem{corollary}[theorem]{Corollary}
\newtheorem{lemma}[theorem]{Lemma}
\theoremstyle{remark}

\DeclareMathOperator{\GL}{GL}

\DeclareMathOperator{\Aut}{\mathrm{Aut}}
\DeclareMathOperator{\Hom}{Hom}

\DeclareMathOperator{\Der}{Der}
\DeclareMathOperator{\ad}{Ad}

\DeclareMathOperator{\Ima}{Im}
\DeclareMathOperator{\id}{id}
\DeclareMathOperator{\codim}{codim}
\DeclareMathOperator{\mult}{mult}

\DeclareMathOperator{\sgn}{sgn}
\DeclareMathOperator{\tr}{Tr}
\DeclareMathOperator{\ord}{ord}
\DeclareMathOperator{\Fr}{Fr}
\DeclareMathOperator{\End}{End}
\DeclareMathOperator{\Mod}{Mod}

\DeclareMathOperator{\Stab}{Stab}

\DeclareMathOperator{\hh}{HH}
\DeclareMathOperator{\h}{H}

\DeclareMathOperator{\gr}{Gr}
\DeclareMathOperator{\Sym}{Sym}

\DeclareMathOperator{\res}{res}
\DeclareMathOperator{\ind}{ind}
\DeclareMathOperator{\Ind}{Ind}

\DeclareMathOperator{\Tot}{Tot}
\DeclareMathOperator{\Def}{Def}
\DeclareMathOperator{\Ext}{Ext}
\DeclareMathOperator{\dR}{dR}
\DeclareMathOperator{\cl}{cl}
\DeclareMathOperator{\exal}{exal}
\DeclareMathOperator{\C}{C}
\DeclareMathOperator{\N}{N}
\DeclareMathOperator{\op}{op}

\DeclareMathOperator{\aad}{ad}

\DeclareMathOperator{\Conj}{Conj}

\DeclarePairedDelimiter\floor{\lfloor}{\rfloor}
\newcommand{\bb}[1]{\mathbb{#1}}
\newcommand{\mc}[1]{\mathcal{#1}}
\newcommand{\mf}[1]{\mathfrak{#1}}

\newcommand{\OO}{\mathcal{O}}

\newcommand{\pd}[1]{\frac{\partial}{\partial #1}}

\newcommand{\colim}{\varinjlim}
\newcommand{\invlim}{\varprojlim}
\renewcommand{\*}{\ast}

\newcommand{\coor}[1]{{#1}^{\textrm{coor}}}

\newcommand{\HC}{\mc{A}_{n-l, l}^H}
\newcommand{\fl}{\coor{\pi}_{\*}\OO_{\textrm{flat}}(\coor{\mc{N}}\times\HC)}

\newcommand\reallywidehat[1]{\arraycolsep=0pt\relax%
\begin{array}{c}
\stretchto{
  \scaleto{
    \scalerel*[\widthof{\ensuremath{#1}}]{\kern-.5pt\bigwedge\kern-.5pt}
    {\rule[-\textheight/2]{1ex}{\textheight}} 
  }{\textheight} %
}{0.5ex}\\           
#1\\                 
\rule{-1ex}{0ex}
\end{array}
}
\makeatletter
\DeclareFontEncoding{LS1}{}{}
\DeclareFontSubstitution{LS1}{stix}{m}{n}
\DeclareMathAlphabet{\mathcal}{LS1}{stixscr}{m}{n}
\makeatother
\DeclareMathOperator{\sHom}{\mathcal{H\mkern-7mu o\mkern-2.5mu m\mkern-1.5mu}}
\DeclareMathOperator{\sExt}{\mathcal{E\mkern-4.5mu x\mkern-2.5mu t\mkern-1mu}}
\DeclareMathOperator{\sTor}{\mathcal{T\mkern-4.5mu o\mkern-2.5mu r\mkern-1mu}}

\DeclareMathOperator{\sDef}{\mathcal{D\mkern-4mu e\mkern-4.5mu f\mkern-1mu}}
\begin{document}
\title{Sheaves of Twisted Cherednik Algebras as Universal Filtered Formal Deformations}
\author{Alexander Vitanov}
\begin{abstract}
According to a statement by Pavel Etingof, in the special case of an affine variety $X$ with a faithful action by a finite group $G$, the sheaf of (twisted) Cherednik algebras $\mc{H}_{1, c, \psi, X, G}$ with formal parameters $c, \psi$ is a universal formal deformation of $\mc{D}_X\rtimes G$ where $\mc{D}_X$ is the sheaf of differential operators on $X$. In the current note, we generalize Etingof's result to the non-affine case. We prove that for a generic smooth analytic or algebraic variety $X$, the sheaf $\mc{H}_{1, c, \psi, X, G}$ with formal $c$ and  $\psi$ is a universal filtered formal deformation of $\mc{D}_X\rtimes G$. To that aim, we first construct quasi-isomorphisms between the Hochschild (co)chain complex of $\mc{D}_X\rtimes G$ and the $G$-invariant part of the direct sum over all elements $g$ in $G$ of sheaves of holomorphic differential forms on the cotangent bundles of the $g$-fixed point submanifolds in $X$. Finally, we  combine these quasi-isomorphisms with results from the theory of algebraic extensions for sheaves of filtered associative algebras to establish a bijective correspondence between the space of isomorphism classes of filtered infinitesimal deformations of $\mc{D}_X\rtimes G$ and the parameter space of $\mc{H}_{1, c, \psi, X, G}$.  
\end{abstract}
\maketitle
\tableofcontents

\section{Introduction}
In the pioneering work \cite{Eti04}, Pavel Etingof proposed a broad generalization of the notion of a rational Cherednik algebra associated to a complex representation $\mf{h}$ of a finite group $G$ to a smooth algebraic or analytic variety $X$ with a faithful action by a finite group $G$ of automorphisms of $X$. He introduced a sheaf of twisted Cherednik algebras $\mc{H}_{t, c, \psi, X, G}$ on the quotient orbifold $X/G$, where $t$ is a complex number, $c$ are class functions on the set of complex reflections in $G$ and $\psi\in \bb H^2(X, \Omega_X^{\geq1})^G$, and showed  that much of the standard theory of rational Cherednik algebras possesses a natural generalization to the global case. For example, one can derive a global notion of a "Calogero-Moser space" for $X$ by means of the center of the sheaf of Cherednik algebras. From the point of view of deformation theory the central result in his work is \cite[Theorem 2.23]{Eti04} according to which the sheaf of twisted Cherednik algebras $\mc{H}_{1, c, \psi, X, G}$ with formal parameters $c$ and $\psi$ on $X/G$, where $X$ is a complex affine variety, is a universal formal deformation of the skew-group algebra $\mc{D}_X\rtimes G$. Etingof's proof relies solely on the fact that affine varieties are  $D$-affine thanks to which the space of infinitesimal deformation of $\mc{D}_X\rtimes G$ is isomorphic to the second Hochschild cohomology group of the associative algebra of global sections $\Gamma(X, \mc{D}_X\rtimes G)$. This reduces the global statement to the well-known case of associative algebras.  Unfortunately, as smooth analytic and algebraic varieties are not $D$-affine in general, one cannot identify $\mc{D}_X\rtimes G$ with $\Gamma(X, \mc{D}_X\rtimes G)$. That is the reason why the correct version of \cite[Theorem 2.23]{Eti04} in the non-affine global context demands  a fully sheaf-theoretic approach.  
Even though some version of  \cite[Theorem 2.23]{Eti04} in the case of generic smooth algebraic and analytic varieties has always been expected to be true until now, a proper generalization of that result and a corresponding proof thereof have lacked. The current note closes this longstanding gap providing necessary mathematical tools along with a final proof of the fact that the sheaf of twisted Cherednik algebras arises as a universal formal filtered deformation of $\mc{D}_X\rtimes G$ in the global algebraic and analytic case.
\subsection{Outline of the main results}
Section \ref{deftheoryofsfaa} is devoted to the definition of the space $\sDef(\Lambda)$ of first-order deformations of a sheaf of filtered associative algebras $\Lambda$. The naive identification of this space with the set $\exal(\Lambda, \Lambda)$ of locally trivial square-zero extensions of $\Lambda$ by the $\Lambda-\Lambda$ bimodule $\Lambda$ is not meaningful for the purposes of deformation theory because, as noted to the author by Pavel Etingof and Valery Lunts, in  crucial cases such as $\Lambda=\mc{D}_X$ on elliptic curve X, the set $\exal(\Lambda, \Lambda)$ is infinite-dimensional. For sheaves of algebras with infinite dimensional space of infinitesimal formal deformations, it is not clear how to define a meaningful notion of a universal formal deformation. However, it turns out that in the special case of sheaves of filtered associative algebras, we can restrict ourselves to a subspace $\exal(\mc{M}, \Lambda)_f$ of filtered square-zero extensions which in many cases including the relevant for our aims case $\Lambda=\mc{D}_X\rtimes G$ has the advantage of being finite-dimensional. The main result in Section \ref{deftheoryofsfaa} is Theorem \ref{filtdeformationspace} whose proof imitates the proof of \cite[Theorem 3.2]{Gray61} for the case of filtered Hochschild cochains. 
\begin{theorem*}[A]
There is an isomorphism of $\bb C$-vector spaces between $\exal_f(\Lambda, \mc{M})$ and the second hypercohomology group $\bb H^2(X, \sigma_{\geq1}\mc{C}_f^{\bullet}(\Lambda, \mc{M}))$ of the filtered Hochschild cochain complex $\mc{C}_f^{\bullet}(\Lambda, \mc{M})$ of $\Lambda$ with values in $\mc{M}$ brutally truncated at $1$.
\end{theorem*}
Section \ref{universaldeform} is the core of the paper and it is mostly devoted to the explicit computation of the hypercohomology group $\bb H^2(X, \sigma_{\geq1}\mc{C}_f^{\bullet}(\mc{D}_X\rtimes G, \mc{D}_X\rtimes G))$ in Theorem (A). Concretely, Section \ref{sheavesofcalabiyaualg} is by and large a rehash of the well-known theory of Calabi-Yau algebras. Here, we contribute a definition for sheaves of Calabi-Yau algebras. The main result is Proposition \ref{symcalabiyau} in which we give an explicit locally free resolution of $\Sym^{\bullet}(\mc{T}_X)$ in terms of left $\Sym^{\bullet}(\mc{T}_X)^e$-modules which we subsequently  use to prove that $\mc{D}_X$ and $\mc{D}_X\rtimes G$ are sheaves of Calabi-Yau algebras of dimension $2\dim(X)$ (\emph{cf.} Proposition \ref{dxcalabiyau} and Corollary \ref{dxgcalabiyau}). To the best of our knowledge, so far the mathematical literature is lacking in formal proofs of these statements. We use these results to prove that the canonical inclusion of the complex of filtered cochains of $\mc{D}_X\rtimes G$ in the Hochschild cochain complex of $\mc{D}_X\rtimes G$ is a quasi-isomorphism (\emph{cf.} Proposition \ref{caninclqiso}).

Section \ref{tracedensitysubsec} deals with the definition of the so-called generalized trace density morphisms with the help of which we subsequently calculate $\bb H^2(X, \sigma_{\geq1}\mc{C}_f^{\bullet}(\mc{D}_X\rtimes G, \mc{D}_X\rtimes G))$. 
This section is divided in three parts.  Section \ref{formalgeom} is a review of Gelfand-Kazhdan's formal geometry. In Section \ref{constrofdxg} we outline a technique, developed in \cite{Vit19}, which realizes the sheaf of untwisted Cherednik algebras $\mc{H}_{1, c, 0, X, G}$ ($c$ is not necessarily formal) in terms of gluing of special sheaves $\coor{\pi}_*\OO_{\textrm{flat}}(\coor{\mc{N}}\times\mc{A}_{n-l, l}^H)$ of flat sections of algebra bundles over connected components of fixed point submanifolds $X_i^H$, associated to parabolic subgroups $H$ in $G$, restricted to the strata $X_H^i$  in $X$. In Sections \ref{sectdoodo} and \ref{gentdm}, we show how this presentation of the sheaf of untwisted Cherdnik algebras combined with an Engeli-Felder construction \cite{EF08} along the lines of \cite{RT12} and \cite{Vit20} yields the desired generalized trace density morphisms
\begin{equation*}
\chi_i^H: \mc{C}_{\bullet}(\mc{D}_X\rtimes G)\rightarrow (j_{i*}^{H}\pi_{i*}^H\Omega_{T^*X_i^H}^{2n-2l_H^i-\bullet})^G
\end{equation*}
for all parabolic subgroups $H$ in $G$ where $\Omega_{T^*X_i^H}^{\bullet}$ is the complex of holomorphic differential forms on the cotangent bundle $\pi_i^H: T^*X_i^H\rightarrow X_i^H$ to the connected component $X_i^H$ with codimension $l_H^i$ and $j_i^H$ is the canonical inclusion of $X_i^H$ in $X$. Using the Calabi-Yau property of $\mc{D}_X\rtimes G$, proven earlier, we arrive at a similar map for the Hochschild cochain complex of $\mc{D}_X\rtimes G$
\begin{equation*}
\mc{X}_i^H:\mc{C}^{\bullet}(\mc{D}_X\rtimes G, \mc{D}_X\rtimes G)\rightarrow (j_{i*}^{H}\pi_{i*}^H\Omega_{T^*X_i^H}^{\bullet-2l_H^i})^G. 
\end{equation*}
Let  $\langle g\rangle$ be the cyclic subgroup of $G$ generated by the element $g$ in $G$. One of the central results in the whole note is the following theorem (\emph{cf.} Theorem \ref{quasiisomthm} and Corollary \ref{cohomologicaltrdcor})    
\begin{theorem*}[B]
For every choice of a linearly independent trace $\phi_{\langle  g\rangle}^i$ of $\widehat{\mc{D}}_{l_g^i}\rtimes\langle g\rangle$ and for every family of holomorphic  Maurer-Cartan forms $\{\omega_{\alpha}\}$ on the cotangent bundle $T^*X_i^g$ of the connected component of the fixed point submanifold $X_i^g$ with codimension $l_g^i$ with values in $\mc{A}_{n-l_g^i, l_g^i}^{\langle g\rangle}$ the maps
\begin{align}
\oplus_{i, g\in G} \chi_i^g: \mc{C}_{\bullet}(\mc{D}_X\rtimes G)\rightarrow\Big(\oplus_{i,g\in G}j_{i*}^g\pi_{i*}^g\Omega_{T^*X_i^g}^{2n-2l_g^i-\bullet}\Big)^G\\
\label{secqis}
\oplus_{i, g\in G}\mc{X}_i^g:\mc{C}^{\bullet}(\mc{D}_X\rtimes G, \mc{D}_X\rtimes G)\rightarrow\Big(\oplus_{i,g\in G}j_{i*}^g\pi_{i*}^g\Omega_{T^*X_i^g}^{\bullet-2l_g^i}\Big)^G
\end{align} 
are quasi-isomorphisms. Moreover, the induced morphisms at the level of homology and cohomology sheaves are canonical.   
\end{theorem*}

In Section \ref{spfiltinfdef}, by means of quasi-isomorphism \eqref{secqis} from Theorem (B), we compute $\bb H^2(X, \sigma_{\geq1}\mc{C}_f^{\bullet}(\mc{D}_X\rtimes G, \mc{D}_X\rtimes G))$ (\emph{cf}. Proposition \ref{scctrhcc}) and the space of infinitesimal filtered formal deformations $\sDef(\mc{D}_X\rtimes G)_f$ of $\mc{D}_X\rtimes G$ (\emph{cf}. Corollary \ref{spaceoffiltdefs}), respectively. With this at hand, we finally arrive at the main result of the current paper - Theorem \ref{univfiltformdefthm}. In this note we work on analytic varieties but all of the presented results can be proven in the case of smooth algebraic varieties in a similar fashion with the appropriate modifications.  
\begin{theorem*}[C]
Let $X$ be a smooth algebraic variety or a smooth analytic variety equipped with a finite subgroup $G\subset\Aut(X)$ acting faithfully on $X$. The sheaf of twisted Cherednik algebras $\mc{H}_{1, c, \psi, X, G}$ on the quotient orbifold $X/G$ is a universal formal filtered deformation of the skew-group algebra $\mc{D}_X\rtimes G$.
\end{theorem*}  
\subsection{Notation}
From now on untill the end of this paper $X$ will denote a complex analytic manifold of dimension $\dim_{\bb C}X=n$ and equipped with an action by a finite group $G$  of holomorphic automorphisms. A complex reflection is an element $g$ in $G$ such that the fixed point submanifold $X^g$ has a connected component of codimension $1$. Throughout the note we do not insist that $G$ contains complex reflections.         
\section{Deformation Theory of Sheaves of Filtered Associative Algebras}
\label{deftheoryofsfaa}
First, we quickly review well-known aspects of the theory of extension of associative algebras and its relation to  formal deformations. Then, we generalize these concepts to sheaves of associative algebras. 
\subsection{Square-zero extensions of associative algebras and infinitesimal deformations}
Let $k$ be a field of characteristic zero. In the following we denote by $\Lambda$ an associative $k$-algebra with a unit and by $\Lambda^{\op}$ the opposite $k$-algebra. Let $\Lambda^e:=\Lambda\otimes_k\Lambda^{\op}$ denote the enveloping algebra of $\Lambda$. In the following, we use that every left $\Lambda-\Lambda$-bimodule is a left $\Lambda^e$-module. 
\begin{definition}
\label{sqzextensionofalgebras}
An extension of $\Lambda$ is a $k$-algebra $\Gamma$ together with an $k$-algebra epimorphism $\sigma: \Gamma\rightarrow\Lambda$.
\end{definition}
An extension $\sigma:\Gamma\rightarrow\Lambda$ is called \emph{square-zero extension} of $\Lambda$ if its kernel, $\ker\sigma$, satisfies $\ker(\sigma)^2=0$. In that case the left $\Gamma^e$-module $\ker(\sigma)$ acquires a well-defined structure of a left $\Lambda^e$-module by $\lambda\otimes\lambda'j=\gamma j\gamma'$ for every $j\in\ker\sigma$ and every $\lambda,\lambda'\in\Lambda$ so that $\sigma(\gamma)=\lambda$ and $\sigma(\gamma')=\lambda'$ for some $\gamma, \gamma'\in\Gamma$.   
\begin{definition}
\label{sqzextensionofalgerasbybimodules}
Let $A$ be a left $\Lambda^e$-module. A square-zero extension of $\Lambda$ by $A$ is a short exact sequence of $k$-vector spaces
\[\zeta:~0\rightarrow A\xrightarrow{\chi}\Gamma\xrightarrow{\sigma}\Lambda\rightarrow0\]
such that $\Gamma$ is a $k$-algebra, $\sigma: \Gamma\rightarrow\Lambda$ is a square-zero extension of $\Lambda$ and for $A$, considered as a left $\Gamma^e$-module by pullback along $\sigma$, the $k$-vector space isomorphism $\chi: A\rightarrow\ker(\sigma)$ is also an isomorphism of left $\Gamma^e$-modules.     
\end{definition}
Two square-zero extensions $\zeta:~0\rightarrow A\xrightarrow{\chi}\Gamma\xrightarrow{\sigma}\Lambda\rightarrow0$ and $\zeta':~0\rightarrow A\xrightarrow{\chi'}\Gamma'\xrightarrow{\sigma'}\Lambda\rightarrow0$ of $\Lambda$ by a $\Lambda^e$-bimodule $A$ are said to be equivalent if there is a $k$-algebra morphism $\beta: \Gamma\rightarrow\Gamma'$ such that the diagram 
\begin{equation*}
 \begin{tikzcd}[row sep=tiny]
&& \Gamma \arrow[dd, "\beta"]\arrow[dr, "\sigma"] \\
 0\arrow[r]&A \arrow[ur, "\chi"] \arrow[dr, "\chi'"] & &\Lambda\arrow[r]&0  \\
&&\Gamma'\arrow[ur, "\sigma'"]
        \end{tikzcd}
\end{equation*} 
commutes. By the $5$-Lemma, if such a morphism $\beta$ exists, it is an isomorphism of algebras. A square-zero extension $\zeta:~0\rightarrow A\xrightarrow{\chi}\Gamma\xrightarrow{\sigma}\Lambda\rightarrow0$ is $k$-\emph{split} if there is a $k$-linear map $\theta: \Lambda\rightarrow\Gamma$ such that $\sigma\circ\theta=\id_{\Lambda}$. The splitting map $\theta$ defines a $k$-linear isomorphism $\Lambda\oplus A\cong\Gamma$. The pull-back of the multiplication in $\Gamma$ to $\Lambda\oplus A$ yields a product on the direct sum given by  $(\lambda_1, a_1)(\lambda_2, a_2)=(\lambda_1\lambda_2, \lambda_1a_2+a_1\lambda_2+\mu(\lambda_2, \lambda_2))$ for all $(\lambda_1, a_1), (\lambda_2, a_2)\in\Gamma$, where $\mu(\lambda_1, \lambda_2)=\theta(\lambda_1\lambda_2)-\theta(\lambda_1)\theta(\lambda_2)$ is a Hochschild $2$-cocycle in $\hh^2(\Lambda, A)$ called the \emph{factor set} of the $k$-split square-zero extension. The so-defined product gives an isomorphism of $k$-algebras $\Gamma\rightarrow\Lambda\oplus A$ defining an equivalence of square-zero extensions. A square-zero extension of $\Lambda$ by $A$ is called \emph{split} if there is a $k$-algebra morphism $\theta: \Lambda\rightarrow\Gamma$ such that $\sigma\circ\theta=\id_{\Lambda}$. In particular, every split extension is $k$-split. All split extensions of $\Lambda$ by a left $\Lambda^e$-module $A$ are equivalent to the \emph{semi-direct sum}  
\begin{equation*}
0\rightarrow A\rightarrow\Lambda\oplus A\rightarrow\Lambda
\end{equation*}
where $\Lambda\oplus A$ is a $k$-algebra with respect to the product $(\lambda_1, a_1)(\lambda_2, a_2)=(\lambda_1\lambda_2, \lambda_1a_2+a_1\lambda_2)$.  
The set of isomorphism classes of extensions of $\Lambda$ by $A$ is denoted by $\Ext_{k}(\Lambda, A)$. The Baer sum of short exact sequences gives $\Ext_{k}(\Lambda, A)$ the structure of a $k$-vector space with zero element the isomorphism class of split extensions of $\Lambda$ by $A$.

Let $t$ be a central element in $\Lambda$.  
We recall the definition of a first-order deformation of $\Lambda$.
\begin{definition}[Definition 3.10, \cite{DMZ07}]
Let $R$ be an augmented unital ring with an augmentation $\epsilon: R\rightarrow k$. A $R$-deformation of $\Lambda$ is an associative $R$-algebra $B$ together with a $k$-algebra isomorphism $B\otimes_kR\cong\Lambda$.  
\end{definition}
In this note we are primarily interested in \emph{infinitesimal} (\emph{first-order}) \emph{deformations} which are $k[t]/(t^2)$-deformations in the sense of the above definition. These deformations can be equivalently characterized using Hochschild cohomology.
\begin{proposition}
\label{hhtwococycle}
Given a $k[t]/(t)^2$-deformation $B$ of $\Lambda$, there is a unique Hochschild 2-cocycle $\mu_1\in\Hom_k(\Lambda\otimes_k\Lambda, \Lambda)$ such that the $k[t]/(t)^2$- vector space $\Lambda\otimes_kk[t]/(t)^2$, equipped with the product $\lambda_1*\lambda_1=\lambda_1\lambda_2+\mu_1(\lambda_1, \lambda_2)t\mod (t)^2$, becomes a $k[t]/(t)^2$-algebra isomorphic to $B$ as $k[t]/(t)^2$-algebras.
\end{proposition}
\begin{proof}
Proof is verbatim identical to the proof of Theorem $3.15$ in \cite{DMZ07}.  
\end{proof}
Two first-order deformations $B$ and $B'$ of $\Lambda$ are said to be equivalent if the coresponding unique Hochschild 2-cocycles $\mu$ and $\mu'$ are cohomologous, e.g. $[\mu]=[\mu']$ in $\hh^2(\Lambda, \Lambda)$. Denote by $\Def(\Lambda)$ the set of equivalence  classes of infinitesimal deformations $B$ of $\Lambda$. We have the following easy to prove but crucial  identity (\emph{cf}. Section $2$ in \cite{Lun03}).
\begin{lemma}
\label{extisomdef}
There is a one-to-one correspondence between the sets $\Ext_k(\Lambda, \Lambda)$ and $\Def(\Lambda)$.
\end{lemma}
\begin{proof}
Suppose $\zeta:~0\rightarrow \Lambda\xrightarrow{\chi}\Gamma\xrightarrow{\sigma}\Lambda\rightarrow0$ is a representative of an isomorphism class in $\Ext_k(\Lambda, \Lambda)$. Since per assumption $k$ is a field, $\zeta$ is $k$-split. That means that it is equivalent to a square-zero extension 
\[\zeta:~0\rightarrow \Lambda t\xrightarrow{\chi}\Lambda\oplus\Lambda t\xrightarrow{\sigma}\Lambda\rightarrow0\]
where we accounted that $\Lambda\cong\Lambda t$ as $\Lambda^e$-bimodules. Consequently, for all $\lambda_1, \lambda_2\in\Lambda$ we have
\[(\lambda_1, 0)(\lambda_2, 0)=(\lambda_1\lambda_2, \mu(\lambda_1, \lambda_2))\]
where $\mu$ is a Hochschild $2$-cocycle factor set. By Proposition \ref{hhtwococycle}, the factor set $\mu$ determines a unique $k[t]/(t)^2$-algebra structure on $\Lambda\oplus\Lambda t\cong\Lambda\otimes_kk[t]/(t)^2$ which is an infinitesimal deformation of $\Lambda$. Any other representative of the isomorphism class of $\zeta$ is equivalent to $\zeta$ and thus induces a cohomologous factor set. Hence, there is a well-defined map 
\begin{equation}
\label{extdef}
\Ext_k(\Lambda, \Lambda)\rightarrow\Def(\Lambda).
\end{equation}
Conversely, assume that $B$ is a $k[t]/(t)^2$-deformation of $\Lambda$. Again by Proposition \ref{hhtwococycle}, it is uniquely described by a Hochschild $2$-cocycle $\mu\in\Hom_k(\Lambda^{\otimes2}, \Lambda)$ which according to Theorem 3.1 in \cite{Mac75} corresponds to the factor set of a square-zero extension $\zeta$ of $\Lambda$ by $\Lambda$. An equivalent $k[t]/(t)^2$-deformation $B'$ of $\Lambda$ is per definition given by a cohomologous Hochschild $2$-cocycle $\mu'\in\Hom_k(\Lambda^{\otimes2}, \Lambda)$ which in turn defines a square-zero extension of $\Lambda$ by $\Lambda$ which is equivalent to $\zeta$. Thus, there is a well-defined map 
\begin{equation*}
\Def(\Lambda)\rightarrow\Ext_k(\Lambda, \Lambda)
\end{equation*} 
such that it and the map \eqref{extdef} are inverses of each other.  
\end{proof}
The bijectivity between $\Ext_k(\Lambda, \Lambda)$ and $\Def(\Lambda)$ equips the later with the structure of a $k$-vectors space.
\subsection{Square-zero extensions of sheaves of filtered algebras and infinitesimal deformations}
The material in the following section is mostly inspired by the content in \cite{Gray61}. We assume that $k=\bb C$. Let $\Lambda$ be a sheaf of associative $\bb C$-algebras on a (complex) manifold $X$ an let $\mc{A}$ be a left $\Lambda^e$-module. 
\begin{definition}
\label{sqzeroextforsheaves}
A square-zero extension of $\Lambda$ by $\mc{A}$ is a sheaf of $\bb C$-algebras  $\Gamma$ together with an exact sequence of sheaves of $\bb C$-vector spaces 
\begin{equation*}
0\rightarrow\mc{A}\xrightarrow{i}\Gamma\xrightarrow{p}\Lambda\rightarrow0  
\end{equation*}
in which $p$ is a square-zero extension of sheaves of $\bb C$-algebras as in Definition \ref{sqzextensionofalgebras} satisfying the properties of Definition \ref{sqzextensionofalgerasbybimodules}.
\end{definition}
 Two algebra extensions $\Gamma$ and $\Gamma'$ of $\Lambda$ by a $\Lambda^e$-module $\mc{A}$ are said to be equivalent if there is a homomorphism of sheaves of $\bb C$-algebras $\kappa: \Gamma\rightarrow\Gamma'$ such that the diagram
\begin{equation*}
 \begin{tikzcd}[row sep=tiny]
&& \Gamma \arrow[dd, "\kappa"]\arrow[dr, "p"] \\
 0\arrow[r]&\mc{A} \arrow[ur, "i"] \arrow[dr, "i'"] & &\Lambda\arrow[r]&0\\
&&\Gamma'\arrow[ur, "p'"]
        \end{tikzcd}
\end{equation*} 
commutes. In that case, $\kappa$ is an isomorphism of sheaves of $\bb C$-algebras by the $5$-Lemma. Unlike the case of square-zero extensions of associative $k$-algebras, discussed in the previous section, in general not every zero-square extension of a sheaf of $\bb C$-algebras $\Lambda$ by a $\Lambda^e$-modules splits as a short exact sequence of sheaves of $\bb C$-vector spaces. In fact, the existence of such splittings is not guaranteed even locally. 
As infinitesimal deformations of sheaves of algebras always arise from (local) splitting morphisms at the level of sheaves of vector spaces, we restrict our attention to a  subclass of so-called \emph{locally trivial} square-zero extensions.
\begin{definition}
\label{loctrivsqrzeroext}
A square-zero extension $0\rightarrow\mc{A}\xrightarrow{i}\Gamma\xrightarrow{p}\Lambda\rightarrow0$ of $\Lambda$ by $\mc{A}$ 
is called locally trivial if there is an open cover $\{U_{\alpha}\}$ of $X$ such that for every $U_{\alpha}$, the short exact sequence of sheaves of $\bb C$-vector spaces 
\[0\rightarrow\mc{A}|_{U_{\alpha}}\xrightarrow{i}\Gamma|_{U_{\alpha}}\xrightarrow{p}\Lambda|_{U_{\alpha}}\rightarrow0\]
is $\bb C_X|_{U_{\alpha}}$-split. i.e. there exists a $\bb C_{X| U_{\alpha}}$-linear homomorphism $j_{\alpha}: \Lambda|_{U_{\alpha}}\rightarrow\Gamma|_{U_{\alpha}}$ such that $p\circ j_{\alpha}=\id_{\Lambda|_{U_{\alpha}}}$. 
\end{definition}   
A square-zero extension $0\rightarrow\mc{A}\xrightarrow{i}\Gamma\xrightarrow{p}\Lambda\rightarrow0$ of $\Lambda$ by $\mc{A}$ is \emph{split} if it admits a morphism of sheaves of algebras $j: \Lambda\rightarrow\Gamma$ such that $p\circ j=\id_{\Gamma}$. In particular, a split square-zero extension is locally trivial. We denote the set of isomorphism classes of locally-trivial  square-zero algebra extensions of $\Lambda$ by $\mc{A}$ by $\exal(\Lambda, \mc{A})$. It is a $\bb C$-vector space with respect to the Baer sum and the isomorphism class of split square-zero extensions of $\Lambda$ by $\mc{A}$ as the zero element. 

For the study of differential operators, it makes sense to restrict our attention to the subcategory of sheaves of  filtered associative $\bb C$-algebras. In the remainder of this section $\Lambda$ will be equipped with an increasing filtration $\Lambda_0\subseteq\Lambda_1\subseteq\dots\subseteq\Lambda_n\subseteq\dots$ which is exhaustive, i.e. $\Lambda=\cup_{i=0}^{\infty}\Lambda_i$. Similarly, the $\Lambda^e$-module $\mc{A}$ will be endowed with an increasing filtration $\mc{A}_0\subseteq\mc{A}_1\subseteq\dots\subseteq\mc{A}_n\subseteq\dots$ with $\Lambda_m\cdot\mc{A}_n\subseteq\mc{A}_{m+n}$ which is exhaustive, i.e. $\mc{A}=\cup_{i=0}^{\infty}\mc{A}_i$. We adapt Definition \ref{sqzeroextforsheaves} and Definition \ref{loctrivsqrzeroext} to that context.  
\begin{definition} 
A filtered square-zero extension of $\Lambda$ by $\mc{A}$ is a square-zero extension 
 \begin{equation*}
\xi:  0\rightarrow\mc{A}\xrightarrow{i}\Gamma\xrightarrow{p}\Lambda\rightarrow0
 \end{equation*}
such that  $\Gamma$ is a sheaf of $\bb C$-algebras with an increasing exhaustive filtration and $i$ and $p$ are filtered morphisms of $\bb C$-vector spaces.

The filtered square-zero extension $\xi$ is locally trivial if there is an open cover $\{U_{\alpha}\}$ of $X$ such that for each $U_{\alpha}$, the short exact sequence 
\[\xi: 0\rightarrow\mc{A}|_{U_{\alpha}}\xrightarrow{i}\Gamma|_{U_{\alpha}}\xrightarrow{p}\Lambda|_{U_{\alpha}}\rightarrow0\]
has a filtration-preserving $\bb C_{X|U_{\alpha}}$-linear homomorphism $j_{\alpha}: \Lambda|_{U_{\alpha}}\rightarrow\Gamma|_{U_{\alpha}}$ such that $p\circ j_{\alpha}=\id_{\Lambda|_{U_{\alpha}}}$.
\end{definition}
A filtered square-zero extension $0\rightarrow\mc{A}\xrightarrow{i}\Gamma\xrightarrow{p}\Lambda\rightarrow0$ of $\Lambda$ by $\mc{A}$ is \emph{split} if it admits a morphism of sheaves of $\bb C$-algebras $j: \Lambda\rightarrow\Gamma$ such that $p\circ j=\id_{\Gamma}$. In particular, a  split filtered square-zero extension is locally trivial. Two filtered square-zero algebra extensions $\Gamma$ and $\Gamma'$ of $\Lambda$ by a $\Lambda^e$-module $\mc{A}$ are said to be \emph{equivalent} if there is a filtered homomorphism of sheaves of $\bb C$-algebras $\kappa: \Gamma\rightarrow\Gamma'$ such that the diagram
\begin{equation*}
 \begin{tikzcd}[row sep=tiny]
&& \Gamma \arrow[dd, "\kappa"]\arrow[dr, "p"] \\
 0\arrow[r]&\mc{A} \arrow[ur, "i"] \arrow[dr, "i'"] & &\Lambda\arrow[r]&0          \\
&&\Gamma'\arrow[ur, "p'"]
        \end{tikzcd}
\end{equation*} 
commutes. The subspace of $\exal(\Lambda, \mc{A})$ comprised of filtered locally trivial square-zero extensions of $\Lambda$ by $\mc{A}$ is denoted by $\exal_f(\Lambda, \mc{A})$. 

From now on, assume that $\Lambda$ and $\mc{A}$ are sheaves of locally convex associative $\bb C$-algebras and left $\Lambda^e$-modules, respectively. If $\hat{\otimes}$ designates the topologically completed  tensor product, let 
\begin{equation}
\label{hochschildchaincomplexsheaf}
\mc{C}_{\bullet}(\Lambda, \mc{A}):=\mc{A}\otimes_{\bb C}\Lambda^{\hat{\otimes}\bullet}
\end{equation}
denote the bounded below Hochschild chain complex of 
$\Lambda$ with coefficients in the $\Lambda-\Lambda$-bimodule $\mc{A}$ and by
\begin{equation}
\label{hochschildcochaincomplexsheaf}
\mc{C}^{\bullet}(\Lambda, \mc{A}):=\sHom_{\Lambda^e}(\Lambda^{\hat{\otimes}\bullet+2}, \mc{A})\cong\sHom_{\bb{C}}(\Lambda^{\hat{\otimes}\bullet}, \mc{A}),
\end{equation}
where $\mc{A}^e:=\mc{A}\otimes_{\mathbb{C}}\mc{A}^{\op}$, the bounded below complexes of continuous Hochschild cochains of $\Lambda$ with values in $\mc{A}$, respectively. Accordingly, we denote the homology sheaf of \eqref{hochschildchaincomplexsheaf} by $\mc{HH}_{\bullet}(\Lambda, \mc{A})$ and the cohomology sheaf of \eqref{hochschildcochaincomplexsheaf} by $\mc{HH}^{\bullet}(\Lambda, \mc{A})$, respectively. If we view $\Lambda^{\hat{\otimes}\bullet+2}$ as an  acyclic bar resolution of $\Lambda$ in terms of left $\Lambda^e$-modules, we can express $\mc{HH}_{\bullet}(\Lambda, \mc{A})$ and $\mc{HH}^{\bullet}(\Lambda, \mc{A})$ as left and right derived functors 
\begin{align*}
&\mc{HH}_{\bullet}(\Lambda, \mc{A})=\sTor_{\bullet}^{\Lambda^e}(\mc{A}, \Lambda)=\h_{\bullet}(\mc{A}\otimes_{\Lambda^e}^L\Lambda)\\
&\mc{HH}^{\bullet}(\Lambda, \mc{A})=\sExt_{\Lambda^e}^{\bullet}(\Lambda, \mc{A})=\h^{\bullet}(R\sHom_{\Lambda^e}(\Lambda, \mc{A})).
\end{align*}
In case that $\Lambda$ and $\mc{A}$ are equipped with an increasing filtration, there is a natural increasing filtration on the $n$-fold tensor product $\Lambda^{\otimes n}$ for every integer $n\geq0$ by 
\[F^p(\Lambda^{\otimes n})=\oplus_{i_{1}+\cdots i_n=p}F^{i_1}\Lambda\otimes\cdots\otimes F^{i_n}\Lambda.\] 
The restriction of the Hochschild cochains of $\Lambda$ with values in $\mc{A}$ to filtration preserving cochains yields a subcomplex $\mc{C}_f^{\bullet}(\Lambda, \mc{A})$ of $\mc{C}^{\bullet}(\Lambda, \mc{A})$ which is equipped with a canonical decreasing filtration 
\[\mc{C}_f^{\bullet}(\Lambda, \mc{A})=F^0\mc{C}_f^{\bullet}(\Lambda, \mc{A})\supset F^1\mc{C}_f^{\bullet}(\Lambda, \mc{A})\supseteq\cdots\supseteq F^p\mc{C}_f^{\bullet}(\Lambda, \mc{A})\supseteq\cdots
\] 
given by
\begin{equation*}
F^p\mc{C}_f^{n}(\Lambda, \mc{A}):=\big\{f\in\sHom_{\bb C}(\Lambda^{\otimes n}, \mc{A}): f(F^q(\Lambda^{\otimes n}))\subseteq F^{q-p}\mc{A} ~~\textrm{for every $q$}\big\}
\end{equation*}
for every $n\geq0$. We denote the cohomology sheaf of $\mc{C}_f^{\bullet}(\Lambda, \mc{A})$ by $\mc{HH}_f^{\bullet}(\Lambda, \mc{A})$. 

The proof of the following theorem is almost a verbatim repetition of the proof of the first half of \cite[Theorem $3.2$]{Gray61} adapted to the filtered case. Nevertheless, for the sake of completeness, we shall provide a detailed proof.
\begin{theorem}
\label{filtdeformationspace}
There is an isomorphism of $\bb C$-vector spaces 
\[\exal_f(\Lambda, \mc{A})\cong\bb H^2(X, \sigma_{\geq1}\mc{C}_f^{\bullet}(\Lambda, \mc{A})).\]
\end{theorem}
\begin{proof}
Let $K^{\bullet,\bullet}:=\check{\C}^{\bullet}\big(\mc{U}, \mc{C}_f^{\bullet}(\Lambda, \mc{A})\big)$ be the \v Cech double complex associated to $\mc{U}$ :
\begin{equation*}
\begin{tikzcd}
&\vdots&&\vdots&\vdots\\
 \cdots \arrow[r, "d"]&K^{i+1, 0}\arrow[u, "\delta"]\arrow[r, "d"]&\cdots\arrow[r, "d"]&K^{i+1, j}\arrow[u, "\delta"] \arrow[r, "d"]
    \arrow[r, "d"]&K^{i+1, j+1}\arrow[u, "\delta"]\arrow[r, "d"]&\cdots \\
 \cdots \arrow[r, "d"]&K^{i, 0}\arrow[u, "\delta"]\arrow[r, "d"]&\cdots\arrow[r, "d"]&K^{i, j}\arrow[u, "\delta"] \arrow[r, "d"]&K^{i, j+1}\arrow[u, "\delta"] \arrow[r, "d"]&\cdots\\
 &\vdots\arrow[u, "\delta"]&&\vdots\arrow[u, "\delta"]&\vdots\arrow[u, "\delta"]
 \end{tikzcd}
\end{equation*} 
in which $\delta$ is the \v Cech differential and $d$ is the standard Hochschild differential with $\delta d-d\delta=0$ whose total complex $\Tot^{\bullet}(K^{\bullet\bullet})$ has a differential $D':=\delta_{p,q}+(-1)^pd_{p,q}$ in bidegree $(p, q)$ and differential $D'_{\textrm{tot}}=\sum_{p+q=n}\delta_{p,q}+(-1)^pd_{p,q}$ in total degree $n$. Let $0\rightarrow\mc{A}\xrightarrow{i}\Gamma\xrightarrow{p}\Lambda\rightarrow0$ be an element in $\exal_f(\Lambda, \mc{A})$. Then by definition there is an open cover $\mc{U}=\{U_{\alpha}\}$ of $X$ together with filtered morphisms of sheaves of vector spaces $j_{\alpha}: \Lambda|_{U_{\alpha}}\rightarrow\Gamma|_{U_{\alpha}}$ for every $\alpha$. On intersection $U_{\alpha\beta}:=U_{\alpha}\cap U_{\beta}$, define the filtered map $h_{\alpha\beta}:=j_{\beta}-j_{\alpha}$. As $p\circ h_{\alpha\beta}=0$, it follows that $\Ima(h_{\alpha\beta})=\Ima(i)\cong\mc{A}$. Hence, abusing notation we obtain a morphism of sheaves of $\bb C$-vector spaces $h_{\alpha\beta}: \Lambda|_{U_{\alpha\beta}}\rightarrow\mc{A}|_{U_{\alpha\beta}}$. Let $\delta$ denote the \v Cech differential in the \v Cech complex $\check{\C}^{\bullet}(\mc{U}, \mc{C}_f^1(\Lambda, \mc{A}))$. Then, 
\begin{align}
\label{cechdiffheqzero}
\delta_{1,1}(h)_{\alpha\beta\gamma}
=(j_{\gamma}-j_{\beta})-(j_{\gamma}-j_{\alpha})+(j_{\beta}-j_{\alpha})=0
\end{align} 
implies that $h_{\alpha\beta}$ is a \v Cech $1$-cocycle in $\check{\C}^{1}(\mc{U}, \mc{C}_f^1(\Lambda, \mc{A}))$. Define a map $f_{\alpha}: \Lambda|_{U_{\alpha}}\otimes_{\mathbb C}\Lambda|_{U_{\alpha}}\rightarrow\mc{A}|_{U_{\alpha}}$ as the composition
\begin{align*}
& \Lambda|_{U_{\alpha}}\otimes_{\mathbb C}\Lambda|_{U_{\alpha}}\xrightarrow{j_{\alpha}\otimes_{\bb C}j_{\alpha}}\Gamma|_{U_{\alpha}}\otimes_{\bb C}\Gamma|_{U_{\alpha}}\xrightarrow{\mult}\Gamma|_{U_{\alpha}}\xrightarrow{q_{\alpha}}\mc{A}|_{U_{\alpha}}
\end{align*}
where $\mult$ denotes the product in $\Gamma$ and $q_{\alpha}:=\id_{\Gamma}-j_{\alpha}\circ p$. The associativity of the product in $\Gamma$ implies that \begin{equation}
\label{hhdifffeqzero}
d_{0,2}f_{\alpha}=0.
\end{equation}
As explained in the proof of \cite[Theorem 3.2]{Gray61}, substituting in the definition of $f_{\alpha}$ the identities $j_{\beta}=j_{\alpha}-h_{\alpha\beta}$, $q_{\beta}=q_{\beta}+h_{\alpha\beta}\circ p$ and $p\circ\mult=\mult(p\otimes p)$ yields 
\begin{align*}
(\delta_{0, 2} f)_{\alpha\beta}&=f_{\beta}-f_{\alpha}\\
 &=h_{\beta\alpha}-h_{\alpha\beta}\\
&=d_{1,1}h_{\alpha\beta}
\end{align*}
which together with Equality \eqref{cechdiffheqzero} and Equality \eqref{hhdifffeqzero} implies 
\[D'_{\textrm{tot}}(f_{\alpha}\oplus h_{\alpha\beta})=(\delta_{1,1}-d_{1,1})h_{\alpha\beta}+(\delta_{0,2}+d_{0,2})f_{\alpha}=0.\] 
This means that $f_{\alpha}\oplus h_{\alpha\beta}$ is a $2$-cocycle in $\Tot^{\bullet}(F^1K^{\bullet\bullet})$. 
Suppose $0\rightarrow\mc{A}\xrightarrow{i'}\Gamma'\xrightarrow{p'}\Lambda$ is an equivalent extension such that the diagram
\begin{equation*}
 \begin{tikzcd}[row sep=tiny]
&& \Gamma \arrow[dd, "k"]\arrow[dr, "p"] \\
 0\arrow[r]&\mc{A} \arrow[ur, "i"] \arrow[dr, "i'"] & &\Lambda\arrow[r]&0          \\
&&\Gamma'\arrow[ur, "p'"]
        \end{tikzcd}
\end{equation*} 
commutes and with associated induced $2$-cocycle $h'_{\alpha\beta}\oplus f'_{\alpha}$ in $\Tot^{\bullet}(F^1K^{\bullet\bullet})$. If we identify $\mc{A}$ with the images of the identity inclusions $i(\mc{A})$ and $i'(\mc{A})$ then $k|_{\mc{A}}=\id_{\mc{A}}$. Let $j_{\alpha}$ and $j'_{\alpha}$ be the corresponding filtered splitting homomorphisms of the filtered square-zero extensions. Set $t_{\alpha}:=j'_{\alpha}-k\circ j_{\alpha}$. Then we have,
\[p'\circ t_{\alpha}=\id_{\Lambda}|_{U_{\alpha}}-\id_{\Lambda}|_{U_{\alpha}}=0\]  
which implies $\Ima(t_{\alpha})\cong\Ima(i')\cong\mc{A}$. Hence, abusing notation this map extends to a linear morphism $t_{\alpha}: \Lambda|_{U_{\alpha}}\rightarrow\mc{A}|_{U_{\alpha}}$, i.e. $t_{\alpha}\in\check{\C}^0(\mc{U}, \mc{C}_f^1(\Lambda, \mc{A}))$. Let $h_{\alpha\beta}:=j_{\beta}-j_{\alpha}$, and $h'_{\alpha\beta}:=j'_{\beta}-j'_{\alpha}$. Note that $k\circ h_{\alpha\beta}=h_{\alpha\beta}$. Then, we have,
\begin{align}
\label{diffcechcoboundary}
h'_{\alpha\beta}-h_{\alpha\beta}&= h'_{\alpha\beta}-k\circ h_{\alpha\beta}\nonumber\\
&=(j'_{\beta}-k\circ j_{\beta})-(j'_{\alpha}-k\circ j_{\alpha})\nonumber\\
&=\delta(t_{\alpha}).
 \end{align}
Consider 
the difference of both $(0,2)$-cochains $f_{\alpha}$ and $f'_{\alpha}$  
\begin{align}
\label{diffhhcoboundary}
\big(f'_{\alpha}-f_{\alpha}\big)(\lambda_1, \lambda_2)&=q'_{\alpha}\circ\mult'(j'_{\alpha}(\lambda_1)\otimes j'_{\alpha}(\lambda_2))-q_{\alpha}\circ\mult(j_{\alpha}(\lambda_1)\otimes j_{\alpha}(\lambda_2))\nonumber\\
&=(\id_{\Gamma'}-j'_{\alpha}\circ p')(j'_{\alpha}(\lambda_1)j'_{\alpha}(\lambda_2))-(\id_{\Gamma}-j_{\alpha}\circ p)(j_{\alpha}(\lambda_1)j_{\alpha}(\lambda_2))\nonumber\\
&=j'_{\alpha}(\lambda_1)j'_{\alpha}(\lambda_2)-j'_{\alpha}(\lambda_1\lambda_2)+j_{\alpha}(\lambda_1\lambda_2)-j_{\alpha}(\lambda_1)j_{\alpha}(\lambda_2)\nonumber\\
&=\big(j'_{\alpha}(\lambda_1)j'_{\alpha}(\lambda_2)-j'_{\alpha}(\lambda_1)kj_{\alpha}(\lambda_2)\big)+j'_{\alpha}(\lambda_1)kj_{\alpha}(\lambda_2)-j'_{\alpha}(\lambda_1)j_{\alpha}(\lambda_2)\nonumber\\
&+\big(j'_{\alpha}(\lambda_1)j_{\alpha}(\lambda_2)-kj_{\alpha}(\lambda_1)j_{\alpha}(\lambda_2)\big)-j'_{\alpha}(\lambda_1\lambda_2)+kj_{\alpha}(\lambda_1\lambda_2)\nonumber\\
&=j'_{\alpha}(\lambda_1)t_{\alpha}(\lambda_2)+t_{\alpha}(\lambda_{1})j_{\alpha}(\lambda_2)-t_{\alpha}(\lambda_1\lambda_2)\nonumber\\
&=\lambda_1\cdot t_{\alpha}(\alpha_2)-t_{\alpha}(\lambda_1\lambda_2)+t_{\alpha}(\lambda_1)\cdot\lambda_2\nonumber\\
&=dt_{\alpha}(\lambda_1, \lambda_1).
\end{align}
We used that $k|_{\mc{A}}=\id_{\mc{A}}$ and in the second to the last line that the extension is square-zero, hence $j_{\alpha}(\lambda_1)aj'_{\alpha}(\lambda_1)=\lambda_1\cdot a\cdot\lambda_2$ for all sections $\lambda_1, \lambda_2$ of $\Lambda$ and every section $a$ of $\mc{A}$. Identities \eqref{diffcechcoboundary} and \eqref{diffhhcoboundary} imply that both $2$-cocycles $h_{\alpha\beta}\oplus f_{\alpha}$ and $h'_{\alpha\beta}\oplus f'_{\alpha}$ differ by a coboundary $D'_{\textrm{tot}}t_{\alpha}=(\delta+d)t_{\alpha}$ in $\Tot^{\bullet}(F^1K^{\bullet\bullet})$. Hence, there is a well-defined mapping 
\begin{equation}
\label{exalch}
\exal_f(\Lambda, \mc{A})\rightarrow \check{\h}^2(\mc{U}, \mc{C}_f^{\bullet}(\Lambda, \mc{A})).
\end{equation}  
Conversely, suppose $f_{\alpha}\oplus h_{\alpha\beta}$ is a $2$-cocycle in $\Tot^{\bullet}(F^1K^{\bullet\bullet})$. Then we can define a locally trivial filtered square-zero $\bb C$-vector space extension of $\Lambda$ by $\mc{A}$ as laid out in \cite[Proposition 3.1]{Gray61}. Let $\Gamma$ be the sheaf defined by
\begin{align*}
\bigcup_{\alpha}(\Lambda\oplus\mc{A})|_{U_{\alpha}}\bigg/\bigg\langle (\lambda, a)|_{U_{\alpha}}\sim(\lambda, a+h_{\alpha\beta}(\lambda))|_{U_{\alpha}}~:~(\lambda, a)|_{U_{\alpha}}\in(\Lambda\oplus\mc{A})|_{U_{\alpha}}\bigg\rangle.
\end{align*} 
The product in $\Gamma|_{U_{\alpha}}$ is given by 
\[(\lambda_1, a_1)|_{U_{\alpha}}\cdot(\lambda_2,  a_2)|_{U_{\alpha}}=(\lambda_1\lambda_2,  \lambda_1a_2+a_1\lambda_2+f_{\alpha}(\lambda_1, \lambda_2))|_{U_{\alpha}}.\] 
The well-definition of the product on intersections $U_{\alpha\beta}$ follows from the compatibility relation for $f_{\alpha}$. With this multiplication the above construction becomes in fact an algebra extension. Assume that $h'_{\alpha\beta}\oplus f'_{\alpha}$ is a cohomologous $2$-cocycle in $\Tot^{\bullet}(F^1K^{\bullet\bullet})$. Then, there is by definition a $(0,1)$-cochain $t_{\alpha}$ such that $h'_{\alpha\beta}\oplus f'_{\alpha}-h_{\alpha\beta}\oplus f_{\alpha}=D'_{\textrm{tot}}t_{\alpha}$. Consequently, we can write for the associated locally trivial filtered square-zero extension $\Gamma'$ of $\Lambda$ by $\mc{A}$
\begin{align*}
\bigcup_{\alpha}(\Lambda\oplus\mc{A})|_{U_{\alpha}}\bigg/\bigg\langle (\lambda, a)|_{U_{\alpha}}\sim(\lambda, a+h_{\alpha\beta}(\lambda)+\delta(t_{\alpha})(\lambda))|_{U_{\alpha}}~:~(\lambda, a)|_{U_{\alpha}}\in(\Lambda\oplus\mc{A})|_{U_{\alpha}}\bigg\rangle
\end{align*} 
We define a morphism of sheaves of $\bb C$-modules $k: \Gamma\rightarrow\Gamma'$ by 
\[(\lambda, a)|_{U_{\alpha}}\mapsto(\lambda, a+t_{\alpha}(\lambda))|_{U_{\alpha}}.\] 
This map is in fact a morphism of sheaves of $\bb C$-algebras since we have 
\begin{align*}
k\big((\lambda_1, a_1)|_{U_{\alpha}}(\lambda_2, a_2)|_{U_{\alpha}}\big)&=k\big((\lambda_1\lambda_1, \lambda_1a_2+a_1\lambda_2+f_{\alpha}(\lambda_1, \lambda_2))|_{U_{\alpha}}\big)\\
&=(\lambda_1\lambda_2, \lambda_1a_2+a_1\lambda_2+f_{\alpha}(\lambda_1, \lambda_2)+t_{\alpha}(\lambda_1\lambda_2))|_{U_{\alpha}}\\
&=(\lambda_1\lambda_2, \lambda_1a_2+a_1\lambda_2+dt_{\alpha}(\lambda_1, \lambda_2)+t_{\alpha}(\lambda_1\lambda_2)+f'_{\alpha}(\lambda_1, \lambda_2))|_{U_{\alpha}}\\
&=(\lambda_1, a_1+ t_{\alpha}(a_1))|_{U_{\alpha}}\cdot(\lambda_2, a_2+t_{\alpha}(a_2))|_{U_{\alpha}}\\
&=k((\lambda_1, a_1)|_{U_{\alpha}})k((\lambda_2, a_2)|_{U_{\alpha}}).
\end{align*}
As evidently $k\circ i=i'$ and $p'\circ k=p$, the morphism $k$ defines an equivalence relation between the extensions $\Gamma$ and $\Gamma'$. Thus, we obtain a well-defined map
\begin{equation}
\label{chexal}
\check{\h}^2(\mc{U}, \mc{C}_f^{\bullet}(\Lambda, \mc{A}))\rightarrow\exal_f(\Lambda, \mc{A}).
\end{equation} 
It is evident that the morphisms \eqref{exalch} and \eqref{chexal} are the inverses of each other. As direct limits preserve exactness, we conclude $\exal_{f}(\Lambda, \mc{A})\cong\check{\h}^2(X, \mc{C}_f^{\bullet}(\Lambda, \mc{A}))$. The claim follows from the fact that $X$ is paracompact by assumption.  
\end{proof}
In the spacial case of $\mc{A}=\Lambda$, inspired by the isomorphism in Lemma \ref{extisomdef}, we define the space of first-order deformations $\sDef(\Lambda)$ of $\Lambda$ as $\exal(\Lambda, \Lambda)$ and the space of filtered first-order deformations $\sDef_f(\Lambda)$ of $\Lambda$ by $\exal_f(\Lambda, \Lambda)$. 
\section{Universal deformation of $\mc{D}_X\rtimes G$}
\label{universaldeform}
\subsection{Sheaves of Calabi-Yau algebras}
\label{sheavesofcalabiyaualg}
Calabi-Yau algebras were introduced and first studied in \cite{Gin06}. These  algebras arise naturally in the theory of non-commutative  deformation of spaces and possess a number of valauble properties. In this note we are interested in them because of the duality between Hochschild homology and cohomology of Calabi-Yau algebras admit. 
Let us recall the definition of a Calabi-Yau algebra from \cite{Gin06}.  
\begin{definition}[Definition 3.2.3, \cite{Gin06}]
\label{calabiyaassualgebra}
An associative $k$-algebra is a Calabi-Yau algbra of dimension $d$ provided that $\Lambda$ has a finitely-generated projective $\Lambda-\Lambda$-bimodule resolution and $\hh^{\bullet}(\Lambda, \Lambda\otimes_{k}\Lambda)\cong \Lambda[-d]$ as graded $\Lambda-\Lambda$-bimodules. 
\end{definition}
Calabi-Yau algebras admit a so-called \emph{van den Berg duality} between Hochschild homology and cohomology which is  precisely formulated in the ensuing theorem. 
\begin{theorem}[\cite{VdB98}] 
\label{VdB}
Let $\Lambda$ be a Calabi-Yau associative $\bb C$-algebra of dimension $d$. Then for every left $\Lambda^e$-module $M$ there is a canonical isomorphism $\hh_i(\Lambda, M)\cong\hh^{d-i}(\Lambda, M)$. 
\end{theorem}
The above concepts admits a natural generalization to sheaves of associative $\bb C$-algebras.
\begin{definition}
A sheaf of associative $\bb C$-algebras $\Lambda$ on a complex space $Y$\footnote{A complex analytic manifold and an orbifold are special examples of a complex space.} is Calabi-Yau of dimension $d$  if it admits a locally free left $\Lambda^e$ resolution and $\sExt_{\Lambda^e}^{\bullet}(\Lambda, \Lambda\otimes_{\bb C}\Lambda)\cong\Lambda[-d]$.
\end{definition}
Let $\mc{P}^{\bullet}\rightarrow\Lambda$ be a locally free left $\Lambda^e$ resolution of $\Lambda$. Then, for every $y\in Y$, the stalk $\Lambda_y$ is a Calabi-Yau algebra in the sense of Definition \ref{calabiyaassualgebra}. Indeed, we have 
\begin{align*}
\hh^{\bullet}(\Lambda_y, \Lambda_y\otimes_{\bb C_y}\Lambda_y)&=\Ext_{\Lambda_y^e}(\Lambda_y,  \Lambda_y\otimes_{\bb C_y}\Lambda_y)\\
&=\h^{\bullet}(\sHom_{\Lambda^e|_{U}}(\mc{P}|_U^{\bullet}, \Lambda|_{U}\otimes_{\bb C}\Lambda|_{U}))_y\\
&=\sExt_{\Lambda^e|_{U}}^{\bullet}(\Lambda|_{U}, \Lambda|_{U}\otimes_{\bb C}\Lambda|_{U})_y\\
&=\Lambda_y[-d]
\end{align*}
where $U$ is an open set in $X$ such that $\mc{P}^{\bullet}|_{U}\cong(\Lambda^e|_{U})^{\oplus r}$ for some integer $r>0$. The following important generalization of Theorem \ref{VdB} holds true for sheaves. 
\begin{lemma}
Let $\Lambda$ be a Calabi-Yau sheaf of  associative algebras of dimension $d$ on a complex space $Y$ and let $\mc{A}$ be a left $\Lambda^e$-module. Then $\mc{HH}_i(\Lambda, \mc{A})\cong\mc{HH}^{d-i}(\Lambda, \mc{A})$.  
\end{lemma}
Later we shall need the following fact about the $\OO_X$-algebra $\Sym^{\bullet}(\mc{T}_X)$. 
\begin{proposition}
\label{symcalabiyau}
The sheaf of $\OO_X$-algebras $\Sym^{\bullet}(\mc{T}_X)$ is a sheaf of Calabi-Yau $\bb C$-algebras of dimension $2n$.
\end{proposition}
\begin{proof}
We start by constructing a locally free resolution of $\Sym^{\bullet}(\mc{T}_X)$ of length $2n$ in terms of left $\Sym^{\bullet}(\mc{T}_X)^e$-modules. Put $Y=X\times X$ with the projection $p_1: Y\rightarrow X$ on the first term and $p_2:Y\rightarrow X$ on the second term of $Y$. Let $\delta: X\rightarrow Y, x\mapsto(x, x)$ be the diagonal embedding with a diagonal $\Delta:=\Ima(\delta)$. Since $X$ is a complex analytic manifold, $\Delta$ is closed and hence analytic in $Y$. It, therefore, defines a sheaf of ideals $\mc{I}_{\Delta} \subset\OO_Y$ with a zero set $\Delta$. Let $W$ be an open set in the product topology of $Y$ and let $f_1, \dots, f_n, g_1, \dots, g_n\in \mc{I}_{\Delta|W}$. Let $(x_1, \dots, x_n, y_1, \dots, y_n)$ be local coordinates on $W$. Set a local section $s|_{W}:=\sum_{i=1}^nf_i\pd{x_i}+\sum g_j\pd{y_j}$ of the tangent sheaf $\mc{T}_Y$. It induces a $\OO_{Y|W}$-module homomorphism 
\begin{align*}
s^{*}: \mc{T}_{Y|W}^*&\rightarrow\OO_{Y|W}\\\omega&\mapsto s^*(\omega):=\langle\omega,s\rangle.\end{align*} 
where $\langle\cdot, \cdot\rangle$ is the pairing of 1-forms and vector fields. From the definition of $s^*$ it is evident that $\Ima(s^*)=\mc{I}_{\Delta}$ which identifies $\Delta$ with the zero submanifold of $s$. By the standard theory of Koszul resolutions of zero submanifolds (\emph{cf.} \cite[Chapter IV, Section 2]{FL85}) we get a resolution 
\begin{align}
\label{koszul}
\bigwedge^{2n}\mc{T}_{Y}^*\rightarrow\dots\rightarrow\bigwedge^{1}\mc{T}_{Y}^*\rightarrow\OO_Y\rightarrow\OO_Y/\mc{I}_{\Delta} 
 \end{align}
of $\OO_Y/\mc{I}_{\Delta}$ of length $2n$ on $Y$. We denote by $\Sym^{\bullet}{\mc{T}_X}\boxtimes\Sym^{\bullet}{\mc{T}_X}$ the external tensor product $\OO_Y\otimes_{p_1^{-1}\OO_X\otimes_{\bb C} p_2^{-1}\OO_X}p_1^{-1}\Sym^{\bullet}{\mc{T}_X}\otimes_{\bb C} p_2^{-1}\Sym^{\bullet}{\mc{T}_X}$. It is per definition a locally free $\OO_Y$-module. Hence, it is a flat module over $\OO_Y$. Ergo, tensoring the exact sequence \eqref{koszul} with $\Sym^{\bullet}{\mc{T}_X}\boxtimes\Sym^{\bullet}{\mc{T}_X}$ over $\OO_Y$ yields an exact sequence of $\OO_Y$-modules 
 \begin{align*}
&\Sym^{\bullet}{\mc{T}_X}\boxtimes\Sym^{\bullet}{\mc{T}_X}\otimes_{\OO_Y}\bigwedge^{2n}\mc{T}_{Y}^*\rightarrow\dots\rightarrow\Sym^{\bullet}{\mc{T}_X}\boxtimes\Sym^{\bullet}{\mc{T}_X}\otimes_{\OO_Y}\bigwedge^{1}\mc{T}_{Y}^*\\
 &\rightarrow\Sym^{\bullet}{\mc{T}_X}\boxtimes\Sym^{\bullet}{\mc{T}_X}\rightarrow\Sym^{\bullet}{\mc{T}_X}\boxtimes\Sym^{\bullet}{\mc{T}_X}\otimes_{\OO_Y}\OO_Y/\mc{I}_{\Delta} 
 \end{align*}  
Let $j_{\Delta}: \Delta\hookrightarrow Y$ be the closed embedding and let $\cdot|_{\Delta}$ denote the corresponding restriction of sheaves to the closed analytic submanifold $\Delta$. As the inverse image of $j_{\Delta}$  is an exact functor, we get an exact sequence of $\OO_{Y|\Delta}$-modules
 \begin{align}
 \label{restnewkoszul}
&j_{\Delta}^{-1}(\Sym^{\bullet}{\mc{T}_X}\boxtimes\Sym^{\bullet}{\mc{T}_X})\otimes_{\OO_{Y|\Delta}}\bigwedge^{2n}\mc{T}_{Y|\Delta}^*\rightarrow\dots\rightarrow
j_{\Delta}^{-1}\big(\Sym^{\bullet}{\mc{T}_X}\boxtimes\Sym^{\bullet}{\mc{T}_X}\big)\nonumber\\
&\rightarrow j_{\Delta}^{-1}\big(\Sym^{\bullet}{\mc{T}_X}\boxtimes\Sym^{\bullet}{\mc{T}_X}\big)\otimes_{\OO_{Y|\Delta}}\otimes\OO_{\Delta} 
 \end{align} 
 where $\OO_{\Delta}:=j_{\Delta}^{-1}\big(\OO_Y/\mc{I}_{\Delta}\big)$. The last term in Sequence \eqref{restnewkoszul} can be rewritten as
 \begin{align}
 \label{restrofsym}
&j_{\Delta}^{-1}\big(\Sym^{\bullet}{\mc{T}_X}\boxtimes\Sym^{\bullet}{\mc{T}_X}\big)\otimes_{\OO_{Y|\Delta}}\OO_{\Delta}=\nonumber\\
&=\OO_{Y|\Delta}\otimes_{(p_1j_{\Delta})^{-1}\OO_X\otimes_{\bb C}(p_2j_{\Delta})^{-1}\OO_X}(p_1j_{\Delta})^{-1}\Sym^{\bullet}(\mc{T}_X)\otimes_{\bb C}(p_2j_{\Delta})^{-1}\Sym^{\bullet}(\mc{T}_X)\otimes_{\OO_{Y|\Delta}}\OO_{\Delta}\nonumber\\
&\cong\OO_{Y|\Delta}\otimes_{(p_1j_{\Delta})^{-1}\OO_X\otimes_{\bb C}(p_2j_{\Delta})^{-1}\OO_X}\Sym^{\bullet}(j_{\Delta}^{-1}(p_1^{-1}\mc{T}_X\oplus p_2^{-1}\mc{T}_X))\otimes_{\OO_{Y|\Delta}}\OO_{\Delta}\nonumber\\
&\cong\Sym^{\bullet}(\mc{T}_{Y|\Delta})\otimes_{\OO_{Y|\Delta}}\OO_{\Delta}\nonumber\\
&\cong\Sym^{\bullet}(\mc{T}_{\Delta})
 \end{align}
where $\mc{T}_{\Delta}$ is the tangent sheaf on the diagonal submanifold $\Delta$. As an isomorphism of abelian categories between $\OO_{\Delta}-\Mod$ and $\OO_{X}-\Mod$ the $\OO_{\Delta}$-module pullback $\delta^*$ is exact, too. This way, applying $\delta^*$ on \eqref{restnewkoszul} and plugging \eqref{restrofsym} in \eqref{restnewkoszul}, we obtain an exact sequence of $\OO_X$-modules 
\begin{align}
 \label{uglybimodkoszulresolsym}
 \delta^*\Big(j_{\Delta}^{-1}(\Sym^{\bullet}{\mc{T}_X}\boxtimes\Sym^{\bullet}{\mc{T}_X})& \otimes_{\OO_{Y|\Delta}}\bigwedge^{\bullet}\mc{T}_{Y|\Delta}^*\Big)\rightarrow\delta^*\Big(\Sym^{\bullet}(\mc{T}_{\Delta})\Big). 
 \end{align} 
 which accounting for $\delta^*\Big(j_{\Delta}^{-1}(\Sym^{\bullet}{\mc{T}_X}\boxtimes\Sym^{\bullet}{\mc{T}_X})\cong \Sym^{\bullet}(\mc{T}_X)\otimes_{\bb C}\Sym^{\bullet}(\mc{T}_X)$ becomes the same as
 \begin{align}
 \label{bimodkoszulresolsym}
 \Sym^{\bullet}(\mc{T}_X)\otimes_{\bb C}\Sym^{\bullet}(\mc{T}_X)\otimes_{\delta^{-1}\OO_{Y|\Delta}}\bigwedge^{\bullet}\delta^{-1}\mc{T}_{Y|\Delta}^*\rightarrow\Sym^{\bullet}(\mc{T}_X).
 \end{align} 
The exact sequences of locally free $\Sym^{\bullet}(\mc{T}_X)-\Sym^{\bullet}(\mc{T}_X)$-bimodules \eqref{uglybimodkoszulresolsym}, respectively \eqref{bimodkoszulresolsym} can be seen as the desired Koszul type resolution of $\Sym^{\bullet}(\mc{T}_X)$ of length $2n$ in terms of locally free left $\Sym^{\bullet}(\mc{T}_X)^e$-modules. 
The Calabi-Yau property follows consequently from
\begin{align*}
&\sExt_{\Sym^{\bullet}(\mc{T}_X)^e}^{\bullet}(\Sym^{\bullet}(\mc{T}_X), \Sym^{\bullet}(\mc{T}_X)\otimes_{\bb C}\Sym^{\bullet}(\mc{T}_X)):=\nonumber\\
&\hspace{10em}=\h^{\bullet}\big(\bb R\sHom_{\Sym^{\bullet}(\mc{T}_X)^e}(\Sym^{\bullet}(\mc{T}_X), \Sym^{\bullet}(\mc{T}_X)\otimes_{\bb C}\Sym^{\bullet}(\mc{T}_X)\big)\nonumber\\
&\hspace{10em}\cong\h^{\bullet}\Big(\Sym^{\bullet}(\mc{T}_X)\otimes_{\bb C}\Sym^{\bullet}(\mc{T}_X)\otimes_{\delta^{-1}\OO_{Y|\Delta}}\bigwedge^{\bullet}\delta^{-1}\mc{T}_{Y|\Delta}^*\Big)\nonumber\\
&\hspace{10em}=\Sym^{\bullet}(\mc{T}_X)[-2n]
\end{align*}
which completes the proof of the statement. 
\end{proof}
With the help of Proposition \eqref{symcalabiyau} we show that the sheaf of holomorphic differential operators $\mc{D}_X$ on a complex analytic manifold $X$ is Calabi-Yau of dimension $2n$.  
\begin{proposition}
\label{dxcalabiyau}
$\mc{D}_X$ is a sheaf of Calabi-Yau algebras of dimension $2n$.   
\end{proposition}
\begin{proof}
Since $\mc{D}_X$ is a filtered quantization of $\Sym^{\bullet}(\mc{T}_X)$, the exact sequence \eqref{uglybimodkoszulresolsym} from Proposition \ref{symcalabiyau} can be rewritten in the form 
\begin{align*}
 \delta^*\Big(\gr^{\bullet}\big(j_{\Delta}^{-1}(\mc{D}_X\boxtimes\mc{D}_X)\big) \otimes_{\OO_{Y|\Delta}}\bigwedge^{\bullet}\mc{T}_{Y|\Delta}^*\Big)\rightarrow\delta^*\Big(\gr^{\bullet}(\mc{D}_{\Delta})\Big).
\end{align*}
It is a known result from homological algebra that for any filtered complex $\C_{\bullet}$, exactness of the associated graded of $\C_{\bullet}$ implies exactness of $\C_{\bullet}$. Hence, 
\begin{align}
\label{dxdxkoszulresolution}
 \delta^*\Big(j_{\Delta}^{-1}(\mc{D}_X\boxtimes\mc{D}_X) \otimes_{\OO_{Y|\Delta}}\bigwedge^{\bullet}\mc{T}_{Y|\Delta}^*\Big)\rightarrow\delta^*\Big(\mc{D}_{\Delta}\Big).
\end{align}
is an exact sequence of locally free $\OO_X$-modules. In particular, this acyclic complex is the wanted Koszul type resolution of length $2n$ of $\mc{D}_X$ in terms of locally free left $\mc{D}_X^e$-modules. The Calabi-Yau property follows in an analogous manner to Proposition \ref{symcalabiyau}.  
\end{proof}
An immediate consequence of Proposition \ref{dxcalabiyau} and Proposition \ref{symcalabiyau} is the following result.
\begin{corollary}
\label{dxgcalabiyau}
The sheaves $(\mc{D}_X)^G$ and $\mc{D}_X\rtimes G$ as well as their corresponding associated graded $\Sym_{\OO_X}^{\bullet}(\mc{T}_X)^G$ and $\Sym_{\OO_X}^{\bullet}(\mc{T}_X)\rtimes G$ on $X/G$ are sheaves of Calabi-Yau $\bb C$-algebras of dimension $2\dim_{\bb C}X$.
\end{corollary}
\begin{proof}
Taking the invariants with respect to the action of $G$ is an exact functor $(~\cdot~)^G:\mathfrak{Bimod}(\mc{D}_X)\rightarrow\mathfrak{Bimod}(\mc{D}_X^G)$. Hence, applying $(~\cdot~)^G$ on \eqref{dxdxkoszulresolution} yields a locally free resolution of $(\mc{D}_X)^G$ on $X/G$ in terms of $(\mc{D}_X)^G-(\mc{D}_X)^G$-bimodules of length $2n$. 
With the proper locally free left $\mc{D}_X^e$-module resolution of $\mc{D}_X$  at hand we obtain 
\begin{equation*}
\sExt_{(\mc{D}_X)^{Ge}}^{\bullet}((\mc{D}_X)^G, (\mc{D}_X)^G\otimes_{\bb C}(\mc{D}_X)^G)=(\mc{D}_X)^G[-2n] 
\end{equation*}
which implies that $(\mc{D}_X)^G$ is Calabi-Yau of dimension $2n$. The Morita equivalence between $(\mc{D}_X)^G$ and $\mc{D}_X\rtimes G$ implies the latter statement. The proof of the claim for the associated graded is literally identical. 
 \end{proof}
 \begin{corollary}
 \label{calabiyauqiso}
 Let $\Theta$ be a Hochschild chain $2n$-cycle of $\mc{D}_X\rtimes G$. Then the morphism of cochain complexes 
 \begin{align*}
 \mu: \mc{C}^{\bullet}(\mc{D}_X\rtimes G, \mc{D}_X\rtimes G)&\rightarrow\mc{C}_{2n-\bullet}(\mc{D}_X\rtimes G, \mc{D}_X\rtimes G)\\
 f&\mapsto\Theta\cap f
  \end{align*}
 is a quasi-isomorphism, where $\mc{C}_{2n-\bullet}(\mc{D}_X\rtimes G, \mc{D}_X\rtimes G)$ is viewed as cochain complex by inverting degrees.
 \end{corollary}
 \begin{proof}
 It follows from the Calabi-Yau propert of $\mc{D}_X\rtimes G$ together with the fact that $\mu_*$ is injective and the cohomology groups are finite-dimensional.
 \end{proof}
\subsection{Generalized trace density morphisms}
\label{tracedensitysubsec} 
We denote by $(\mc{A}_{X, \bb C}^{\bullet}, d_{dR})$ the de Rham complex of sheaves of smooth complex-valued differential forms with differential $d_{dR}$, and by $(\Omega_X^{\bullet}, d_{dR})$ the complex of sheaves of holomorphic differential forms on $X$ with differential $d_{dR}$.

In Section \ref{tracedensitysubsec}, we adapt Engeli-Felder's trace density morphism $(2)$ in \cite{EF08}, 
\[\chi: \mc{C}_{\bullet}(\mc{D}_X)\to\mc{A}_X^{2n-\bullet},\] 
to the holomorphic setting. We achieve this by replacing the sheaf $\mc{A}_X^{2n-\bullet}$ with a sheaf of holomorphic differential operators on the cotangent bundle of $X$. Subsequently, making use of a formal geometric realization of $\mc{D}_X\rtimes G$ from the unpublished manuscript \cite{Vit19}, we successfully extend the holomorphic modification of the Engeli-Felder's construction to $\mc{D}_X\rtimes G$. Contrary to the smooth case, in the holomorphic setting we deal with hypercohomology classes which cannot easily be integrated. Hence, in stead of trace density morphism we use in the holomorphic context  the vague term \emph{generalized trace density morphism}.


\subsubsection{Review of Gelfand-Kazhdan's formal geometry}
\label{formalgeom}
For the construction of generalized trace density maps we need to review Gelfand-Kazhdan's formal geometry \cite{GK71}. At its core, the presented material in this section is not original. It is a very succinct summary of many scattered facts, gathered by the author from works such as \cite{GK71}, \cite{BR73},  \cite{CFT02}, \cite{BK04}, \cite{Kho07}, \cite{EF08},  \cite{CV10}, \cite{CRV12}, \cite{shilin15}, \cite{ggw16}, \cite{shilin17} and systematized in \cite{Vit19}. The text here repeats by an large a summary on formal geometry in Section $4.2$ in \cite{Vit20}. 

The isotropy type of a subgroup $H$, denoted $X_H$, is the set of points $x$ in $X$ with stabilizer $\Stab(x)=H$. It is a locally closed submanifold in $X$ with an  embedding $j_H:X_H\hookrightarrow X$. As $G$ is finite, its action defines a stratification of $X$ with strata the connected components of the isotropy types $X_H$ in $X$ for all parabolic subgroups $H$ in $G$. A detailed discussion of stratified spaces can be found in e.g., \cite{OR04}. Every parabolic subgroup $H$ of $G$ acts linearly on the fibers of the tangent bundle $TX|_{X_H}$ of $X$, restricted to $X_H$. This fact induces a canonical splitting $TX|_{X_H}\cong TX_H\oplus\N$ in which $TX_H$ is the tangent bundle to the isotropy type $X_H$ and $\N$ is a subbundle of $TX|_{X_H}$ which is complementary to $TX_H$ and isomorphic to the normal quotient bundle over $X_H$. As the action of $G$ on $X$ is faithful, the subgroup $H$ acts faithfully on the fibers of $\N$. Hence, it is  embedded in $\GL_l(\bb C)$. We denote the centralizer of the image of this embedding by $Z$  and its Lie algebra by $\mf{z}$.

We equip $X$ with the non-Hausdorff so-called $G$-equivariant topology $\mf{T}_X^G$ which is comprised of preimages of open subsets in $X/G$. All of these preimages have the form $\bigcup_{g\in G}gV$ for an open set $V$ in $X$.

Now, we define a basis for the $G$-equivariant topology on $X$. A \emph{slice} at a point $x\in X$ is a $\Stab(x)$-invariant neighborhood $W_x$ such that $W_x\cap gW_x=\varnothing$ for all $g\in G\setminus \Stab(x)$. A slice is called a \emph{linear} if there is a $\Stab(x)$-invariant open set $V$ in $\bb C^n$ such that $W_x$ is $\Stab(x)$-equivariantly biholomorphic to $V$. As the group $G$ is finite, it acts on $X$ properly discontinuously. Hence, each point $x$ in $X$ possesses a slice $W_x$. Furthermore, by Cartan's Lemma one can always shrink the slice $W_x$ until the $\Stab(x)$-action is linearized. Thus, every point in $X$ possesses a fundamental system of linear slices. Moreover, for every member $W_x'$ of a fundamental system of linear slices at $x$ in $X$ with $\Stab(x):=H$, there always exists a finer $H$-invariant linear slice $W_x\subset W_x'$, which is $H$-equivariantly biholomorphic to a polydisc in $\bb C^{n-l}\times\bb C^l$, where $\bb C^{n-l}$ is the fixed point subspace of $\bb C^n$ with respect to $H$. Consequently, $W_x$ is Stein open and $\ind_H^G(W_x):=\sqcup_{g\in G/H}gW_x$ is a subset of $\bigcup_{g\in G}gW_x'$. Hence, the collection $\mf{B}_X^G$ of sets $\ind_H^G(W_x)$ where $W_x$ is either a  $H$-invariant linear slice biholomorphic to a polydisc in $\bb C^{n-l}\times\bb C^l$ or an $H$-invariant open subset of $\mathring{X}$ with $gW_x\cap W_x=\varnothing$ for all $g\in G\setminus H$, forms a basis of G-invariant Stein open sets for $\mf{T}_X^G$.

Note that for each $x\in X_i^H$, every Stein linear slice $W_x$ constitutes a holomorphic slice chart for the fixed point submanifold $X_i^H$. That is, if $x^1, x^2, \dots, x^{n-l}$ are the holomorphic coordinates on the fixed point subspace $\big(\bb C^n\big)^H=\bb C^{n-l}$ and $y^1, \dots, y^l$ are the holomorphic coordinates on the $l$-dimensional complement of $\bb C^{n-l}$ in $\bb C^n$, then $(x^1, \dots, x^{n-l}, y^1, \dots, y^l)$ defines a  holomorphic parametrization of $W_x$ such that $(x^1, \dots, x^{n-l})$ are holomorphic coordinates of $W_x\cap X_i^H$ and $(y^1, \dots y^l)$ are  holomorphic coordinates on $W_x$ in transversal direction to $X_i^H$.

Now, let $\mc{N}$ be the locally free $\OO_{X_H^i}$-module corresponding to the normal bundle $\N$ to a stratum $X_H^i$ with $\codim X_H^i=l$ where $i$ goes through the connected components. We denote by $ \coor{\mc{N}}$ the set of pairs consisting of a closed immersion of $\bb C$-ringed spaces $\Phi_x:=(\varphi, \varphi^{\#}): (0, \widehat{\OO}_{l})\to(X_H, \OO_{X_H})$ with $x=\varphi(0)$ and an isomorphism of $\widehat{\OO}_{n-l}$-modules $f: \widehat{\OO}_{n-l}^{\oplus l}\to\varphi^*\mc{N}$ where $\widehat{\OO}_{l}$ and $\widehat{\OO}_{n-l}$ are the rings of formal functions in a formal neighborhood at the origin of $\bb C^l$ and $\bb C^{n-l}$, respectively. The projection $\coor{\pi}: \coor{\mc{N}}\to X_H^i$, given by $(\Phi_x, f)\mapsto x$, turns $\coor{\mc{N}}$ into a fiber bundle over $X_H^i$ with fiber at $x$ bijective to the set of infinite jets $[\phi]_x$ of parametrizations $\phi:  \bb C^{n-l}\times \bb C^{l}\to N$ at $0$ with $\phi(0, 0)\in\N_x$. Let $\bb G:=\Aut_{n-l}\times Z(\widehat{\OO}_{n-l})$ be the pro-Lie group in which $\Aut_{n-l}:=\Aut(\widehat{\OO}_{n-l})$ and $Z(\widehat{\OO}_{n-l})$ is the group of formal power series in the coordiantes ${\bf x}=(x_1, \dots, x_{n-l})$ of $\bb C^{n-l}$ with coefficients matrices in $Z$. It acts freely and transitively on the fiber of $\coor{\mc{N}}$ from the right which makes $\coor{\mc{N}}$ a principal $\bb G$-bundle. Let $W_{n-l}:=\Der(\widehat\OO_{n-l})$ be the Lie algebra of vector fields in the formal neighboorhhod of $0$ in $\bb C^{n-l}$. According to formal geometry, at every point $(\Phi_x, f)$ in $\coor{\mc{N}}$,  there is an isomorphism between the tangent space $T_{(\Phi_x, f)}\coor{\mc{N}}$ and the Lie algebra semidirect sum $W_{n-l}\rtimes\mf{z}\otimes\widehat{\OO}_{n-l}$. This induces a holomorphic $\bb G$-equivariant connection 1-form $\omega$ with values in $W_{n-l}\rtimes\mf{z}\otimes\widehat{\OO}_{n-l}$ which in the terminology of \cite{BK04} gives $\coor{\mc{N}}$ the structure of a transitive Harish-Chandra $(W_{n-l}\rtimes\mf{z}\otimes\widehat{\OO}_{n-l}, \bb G)$-torsor over $X_H^i$. 

As $\GL_{n-l}(\bb C)\times Z$ is a closed subgroup of $\bb G$, the projection $\bb G\to\bb G/\GL_{n-l}(\bb C)\times Z$ is a principal $\GL_{n-l}(\bb C)\times Z$-bundle. Hence, the projection map
\begin{align}
\label{princbund}
\coor{\mc{N}}\cong\coor{\mc{N}}\times_{\bb G}\bb G\longrightarrow\coor{\mc{N}}\times_{\bb G}(\bb G/(\GL_{n-l}(\bb C)\times Z))\cong\coor{\mc{N}}/(\GL_{n-l}(\bb C)\times Z))
\end{align}
is a principal $\GL_{n-l}(\bb C)\times Z$-bundle. The total space $\coor{\mc{N}}$ of \eqref{princbund} is a homogenous principal $W_{n-l}\rtimes\mf{z}\otimes\widehat{\OO}_{n-l}$-space. As the Lie algebra action commutes with the action of $\GL_{n-l}(\bb C)\times Z$ in a way compatible with the Harish-Chandra pair $(W_{n-l}\rtimes\mf{z}\otimes\widehat{\OO}_{n-l}, \GL_{n-l}(\bb C)\times Z)$, the principal bundle $\coor{\mc{N}}\to\coor{\mc{N}}/(\GL_{n-l}(\bb C)\times Z))
$ is in fact a transitive Harish-Chandra $(W_{n-l}\rtimes\mf{z}\otimes\widehat{\OO}_{n-l}, \GL_{n-l}(\bb C)\times Z)$-torsor.

In the special situation $G=\{\id_G\}$, the only stratum in $X$ is $X$ itself. In that case, the torsor $\coor{\mc{N}}|_X$ reduces to the bundle $\coor{X}$ of formal coordinate systems of $X$, which is a Harish-Chandra $(W_n, \Aut_n)$-torsor. Similarly, the projection map $\coor{X}\to\coor{X}/\GL_n(\bb C)$ becomes a Harish-Chandra $(W_n, \GL_n(\bb C))$-torsor.
\subsubsection{Review of the formal geometric construction of $\mc{D}_X\rtimes G$} 
\label{constrofdxg}
For the definition of generalized trace density morphisms in Sections \ref{sectdoodo} and \ref{gentdm} we need to recollect the main steps in \cite{Vit19} of the construction of sheaves of Cherednik algebras $\mc{H}_{1, c, X, G}$ via formal geometry. A specialization at $c=0$ yields a formal geometric realization of $\mc{D}_X\rtimes G$ which is needed for the generalized trace density morphisms. The review here repeats by an large the exposition in Section $4.3$ of \cite{Vit20} in the special case $c=0$. 
     
We denote by $\mc{A}_{n-l, l}^H$ the algebra $\widehat{\mc{D}}_{n-l}\otimes\widehat{\mc{D}}_{l}\rtimes H$. Now, let $(u_i)$, $(y_i)$ be bases of $\bb C^{l}$ and its dual, respectively. Then, by \cite[Proposition 4.4]{Vit19}, the Lie algebra representation $\aad\circ\Phi: W_{n-l}\rtimes\mf{z}\otimes\widehat{\OO}_{n-l}\to\End(\mc{A}_{n-l, l}^H)$, where $\aad$ is the adjoint action and 
 \begin{align}
 \label{liealghom}
 \Phi: W_{n-l}\rtimes\mf{z}\otimes\widehat{\OO}_{n-l}&\to\mc{A}_{n-l, l}^H\\    
v+A\otimes p&\mapsto v\otimes\id+p\otimes-\sum_{i, j}A_{ij}y_ju_i\nonumber
\end{align}
is an injective Lie algebra homomorphism, gives $\mc{A}_{n-l, l}^H$ the structure of Harish-Chandra $(W_{n-l}\rtimes\mf{z}\otimes\widehat{\OO}_{n-l}, \GL_{n-l}(\bb C)\times Z)$-module. The localization per \cite{BK04} of the Harish-Chandra module $\mc{A}_{n-l, l}^H$ with respect to the torsor  \eqref{princbund} yields a trivial holomorphic $\GL_{n-l}(\bb C)\times Z$-equivariant vector bundle $\coor{\mc{N}}\times\mc{A}_{n-l, l}^H\to\coor{\mc{N}}$ with a flat holomorphic $\GL_{n-l}(\bb C)\times Z$-equivariant connection $\nabla_H^i:=d+\aad\circ\Phi\circ\omega$ with values in $\mc{A}_{n-l, l}^H$. We denote the sheaf of flat sections of that bundle over $X_H^i$ by $\coor{\pi}_*\OO_{\textrm{flat}}(\coor{\mc{N}}\times\mc{A}_{n-l, l}^H)$. Let $\widehat{W}_x$ denote the  completion of $W_x$ with respect to the analytic subset $W_{x, H}^i:=W_x\cap X_H^i$. The first main result in the gluing procedure is the following theorem (\emph{cf.} \cite[Theorem 5.5]{Vit19}).
\begin{theorem}
\label{formalgeomident}
Let $W_x$ be a $H$-invariant linear slice which is Stein. Then, for every parabolic subgroup $H$ of $G$, there is an isomorphism of sheaves of $\bb C$-algebras 
\[ \mc{Y}_H^i: j_{H*}^i\coor{\pi}_*\OO_{\textrm{flat}}(\coor{\mc{N}}\times\mc{A}_{n-l, l}^H)(W_x)\cong \OO_{\widehat{W}_x}(W_{x, H}^i)\otimes_{\OO(W_x)}\mc{D}_X(W_x)\rtimes H.\]
\end{theorem}
Next, we give a method by means of which the sheaves $j_{H*}^i\coor{\pi}_*\OO_{\textrm{flat}}(\coor{\mc{N}}\times\mc{A}_{n-l, l}^H)$ on the strata of various codimensions can be  glued into a single sheaf on $X$ in the $G$-equivariant topology. The stratification defines a finite increasing filtration of $X$ into $G$-invariant open subsets $F^0(X)=\mathring{X}\subset F^1(X)\subset\cdots\subset F^{l_{\max}}(X)=X$ where $\mathring{X}$ is the principal dense open stratum corresponding to the trivial subgroup of $G$, $F^{k}(X)$ is the open disjoint union of strata of codimension less or equal to $k$. For every $k$, we successively define a sheaf $\mc{S}^k$ on $F^k(X)$ by gluing $\mc{S}^{k-1}$ with all $j_{H*}^i\coor{\pi}_*\OO_{\textrm{flat}}(\coor{\mc{N}}\times\mc{A}_{n-l, l}^H)$ on the strata of codimension $k$ in $X$. The gluing is implemented locally on Stein $H$-invariant linear slices. 

We begin by setting $\mc{S}^0:=\mc{D}_{\mathring{X}}\rtimes G$. Next, for every  basic open set $\ind_H^GW_x$ with $x\in X_H^i$ and $\codim(X_H^i)=1$, we define $\mc{S}^1(\ind_H^GW_x)$ as the set of pairs of sections 
\begin{align*}
&\sum_{(g, g')\in G/H\times G/H}g\otimes p_{gg'}\otimes g'\in\bb CG\otimes_{\bb CH}\coor{\pi}_*\OO_{\textrm{flat}}(\coor{\mathring{X}}\times\widehat{D}_n)(W_x\setminus X_H^i)\rtimes H\otimes_{\bb CH}\bb CG\\
&\sum_{(g, g')\in G/H\times G/H}g\otimes s_{gg'}\otimes g'\in\bb CG\otimes_{\bb CH}\coor{\pi}_*\OO_{\textrm{flat}}(\coor{\mc{N}}\times\mc{A}_{n-l, l}^H)(W_{x, H}^i)\otimes_{\bb CH}\bb CG
\end{align*} 
saturating the conditions
\begin{align*}
&1.\quad \mc{Y}(p_{gg'})\in\OO(W_x)[\mc{R}(W_x)^{-1}]\otimes_{\OO(W_x)}\mc{D}_X(W_x)\rtimes H,\quad \textrm{for all}~ (g, g')\in G/H\times G/H,\\
&2.\quad s_{gg'}([\phi]_{\psi(0, 0)})=i_{\psi}( \mc{Y}(p_{gg'}))\quad \textrm{for all}~ (g, g')\in G/H\times G/H, \infty-\textrm{jets}~~[\phi]\in\coor{\mc{N}}. 
 \end{align*}
Here, $\mc{R}(W_x)$ is the multiplicative subset of $\OO(W_x)$ comprised of the constant function $1$ and all functions $f: W_x\to\bb C$ with $f|_{W_x\cap D}=0$ and $f(p)=0$ for all $p\in W_x\setminus D$. Moreover, $\psi$ is a parametrization of $W_x$ with $\psi(0, 0)=\pi(\phi(0, 0))\in W_x$, and $i_{\psi}$ is the Taylor expansion operator with respect to ${\bf x}=0$. By \cite[Proposition 6.1]{Vit19}, the collection of algebras $\mc{S}^{1}(\ind_H^GW_x)$ induces a sheaf $\mc{S}^1$ in the $G$-equivariant topology of $F^1(X)$. By \cite[Proposition 6.2]{Vit19}, there exists  an isomorphism of sheaves of $\bb C$-algebras
\begin{align*}
\mf{X}^1: \mc{S}^1\to\mc{D}_{F^1X}\rtimes G.
\end{align*}
By induction, for every $k$ with $2\leq k\leq l_{\max}$ and every basic open set $\ind_K^GW_x$ with $x\in X_K^j$ and $\codim(X_K^j)=k$, we define $\mc{S}^k(\ind_K^GW_x)$ analogously as the set of pairs of sections
\begin{align*}
q&\in\mc{S}^{k-1}(\ind_K^G(W_x\setminus X_K^j)),\\
\sum_{(g, g')\in G/K\times G/K}g\otimes s_{gg'}\otimes g'&\in\bb CG\otimes_{\bb CK}\coor{\pi}_*\OO_{\textrm{flat}}(\coor{\mc{N}}\times\mc{A}_{n-l, l}^K)(W_{x, K}^i)\otimes_{\bb CK}\bb CG
\end{align*} 
satisfying the gluing condition
\begin{align}
\label{glcontwo}
\sum_{(g, g')\in G/K\times G/K}g\otimes s_{gg'}([\phi]_{\psi(0, 0)})\otimes g'=\sum_{(g, g')\in G/K\times G/K}g\otimes i_{\psi}(d_{gg'})\otimes g'
\end{align}
where $\mf{X}^{k-1}(q)=\sum_{(g, g')\in G/K\times G/K}g\otimes d_{gg'}\otimes g'$ with $d_{gg'}\in\mc{D}_X(W_x\setminus X_K^j)\rtimes K$ according to \cite[Corollary B.8]{Vit19}. This definition induces a presheaf $\mc{S}^k$ in the $G$-equivariant topology on $F^k(X)$. The next result states that $\mc{D}_X\rtimes G$ is realized by a stepwise gluing of the presheaves $\mc{S}^k$ with $j_{H*}^i\coor{\pi}_*\OO_{\textrm{flat}}(\coor{\mc{N}}\times\mc{A}_{n-l, l}^H)$ on the strata in $X$ with codimension $\codim(X_H^i)=k+1$ for all $k=1, \dots, l_{\max}-1$ (\emph{cf}. \cite[Theorem 6.6]{Vit19}). 
\begin{theorem}
\label{maingluingresult}
For every integer $k$ with $1\leq k \leq l_{\max}$, the assignment $\ind_H^GW_x\mapsto\mc{S}^{k}(\ind_H^GW_x)$ defines a sheaf of $\bb C$-algebras $\mc{S}^k$ such that $\mc{S}^k\cong\mc{D} _{F^k(X)}\rtimes G$ as sheaves of $\bb C$-algebra. 
\end{theorem}
The immediate and most important consequence of Theorem \ref{maingluingresult} is that for every stratum $X_H^i$, there is a map of sheaves
\begin{align}
\label{collapsingmap}
p: \mc{D}_X\rtimes G\to j_{H*}^i\coor{\pi}_*\OO_{\textrm{flat}}(\coor{\mc{N}}\times\mc{A}_{n-l, l}^H)
\end{align}
which for brevity we call \emph{collapsing map}. 
Concretely, for every $\ind_H^GW_x$ from $\mc{B}_{X}^G$ with $x\in X_H^i$, the collapsing map is given by
\begin{align}
\label{projonslices}
p: \Gamma (\ind_H^GW_x, \mc{D}_X\rtimes G)&\longrightarrow\coor{\pi}_*\OO_{\textrm{flat}}(\coor{\mc{N}}\times\mc{A}_{n-l, l}^H)(W_{x, H}^i)\\\nonumber  
(q, \sum_{(g, g')\in G/H\times G/H} g\otimes s_{gg'}\otimes g')&\mapsto\id_G\otimes s_{HH}\otimes \id_G
\end{align}
where $t_{HH}$ is the section representing the left coset of the identity in $G$. This is a well-defined map. For a set $\ind_K^GW_x$ in $\mc{B}_{X}^G$ with $x$ lying on a stratum $X_K^j$,  which is contained in the closure of  $X_H^i$, the assignment  
\begin{equation}
\label{projonarbitraryopens}
\Gamma(\ind_K^GW_x, \mc{D}_X\rtimes G)\rightarrow\fl(W_{x, H}^i)
\end{equation}
is defined in a more subtle fashion. By Cartan's Lemma the slice $W_x$ intersects only isotropy types $X_L$  associated to subgroups $L$ of $K$. Hence, by \cite[Corollary B.9]{Vit19}, the map \eqref{projonarbitraryopens} is the composition of the ensuing maps
\begin{align} 
\label{compofres}
&\mc{D}(\ind_K^GW_x)\rtimes G\xrightarrow{\cong} \bb CG\otimes_{\bb CK}\mc{D}_X(W_x)\rtimes K\otimes_{\bb CK}\bb CG\rightarrow \mc{D}_{X}(W_x\setminus X_K^j) \rtimes K\nonumber\\
&\xrightarrow{\cong}\invlim_{y\in X_L\cap W_x}\mc{D}_{X}(\ind_L^KW_y)\rtimes K\nonumber\\
&\twoheadrightarrow\prod_{L<H}\prod_{y\in W_{x, L}^i}\Big\{(s_i)_i\in\prod\mc{D}_{X}(\ind_L^KW_y)\rtimes K:~\res_{\ind_{L_1}^KW_{y'}}^{\ind_{L_{2}}^KW_y}(s_i)=s_j, L_1<L_2<L\Big\}\nonumber\\
&\twoheadrightarrow\prod_{y\in W_{x, H}^i}\Big\{(s_i)_i\in\mc{D}_{X}(\ind_H^KW_y)\rtimes K:~\res_{\ind_{H}^KW_{y'}}^{\ind_{H}^KW_y}(s_i)=s_j\Big\}\nonumber\\
&\xrightarrow{\cong}\prod_{y\in W_{x, H}^i}\Big\{(s_i)\in\prod\bb CK\otimes_{\bb CH}\mc{D}_{X}(W_y)\rtimes H\otimes_{\bb CH}\bb CK:~\res_{W_{y'}}^{W_y}(s_i)=s_j\Big\} \nonumber\\
&\twoheadrightarrow\prod_{y\in W_{x, H}^i}\Big\{(s_i)\in\prod\mc{D}_{X}(W_y)\rtimes H:~\res_{W_{y'}}^{W_y}(s_i)=s_j\Big\}\nonumber\\
&\rightarrow\fl(W_{x, H}^i)
 \end{align}
where $W_y$ is a $H$-invariant linear slice contained in $W_x$. The fourth arrow is onto since $L$-invariant linear slices $W_y$, when $L$ is not contained in $H$,  do not contain $H$-invariant slices. The last arrow in \eqref{compofres} is described the following way. By gluing condition \eqref{glcontwo}, each element in $\mc{D}_X(W_y)\rtimes H$ is uniquely represented by a pair $(q|_{W_y\setminus X_H^i}, s|_{W_{y, H}^i})$ of $\mc{S}^{l}(W_y\setminus X_H^i)\times\fl(W_{y, H}^i)$. For all $y', y''\in W_{x, H}^i$ and Stein $H$-invariant liner slices $W_{y'}, W_{y''}\subset W_x$, the corresponding sections $s'|_{W_{y', H}^i}$ and $s''|_{W_{y'', H}^i}$ coincide on every open sets $W_{y'', H}^i$, contained in $W_{y', H}^i\cap W_{y'', H}^i$. By the the gluing axiom of sheaves, there exists a section $s|_{W_{x, H}^i}$ of $j_{H*}^i\fl$ over $W_{x, H}^i$ which restricts to $s'|_{W_{y', H}^i}$ and $s''|_{W_{y'', H}^i}$, respectively and by the uniqueness axiom of sheaves it is unique. Hence, the assignment \eqref{projonarbitraryopens} is well-defined. 
Thus, the collapsing morphism \eqref{collapsingmap} is a well-defined map of sheaves on the basis $\mc{B}_{X}^G$. After taking projective limit, it becomes a well-defined map of sheaves in the $G$-equivariant topology of $X$, or equivalently on $X/G$. 

\subsubsection{Generalized trace density morphism for $\mc{D}_X$}
\label{sectdoodo}
Let $\omega$ denote the $\Aut_n$-invariant $W_n$-valued holomorphic connection $1$-form on the Harish-Chandra $(W_n, \Aut_n)$-torsor $\coor{X}$, discussed in Section \ref{formalgeom}. It induces a flat holomorphic $\GL_n(\bb C)$-equivariant connection $d+[\omega, \cdot]$ on $\coor{X}\times \widehat{\mc{D}}_n\to\coor{X}$. By Theorem \ref{formalgeomident}, there is an isomorphism of sheaves of $\bb C$-algebras between $\mc{D}_X$ and $\coor{\pi}_*\OO_{\textrm{flat}}(\coor{X}\times\widehat{\mc{D}}_n)$. For the definition of a generalized trace density map for $\mc{D}_X$, it is  convenient to express $\mc{D}_X$ locally in terms of flat sections of an appropriate bundle over the cotangent bundle $T^*X$ rather than over $\coor{X}$. We do that in the following way. 
Let $\mc{U}=\{U_{\alpha}\}$ be a cover of open Stein subsets in $X$ trivializing the cotangent bundle $\pi: T^*X\rightarrow X$, hence the frame bundle $\Fr(X)\rightarrow X$. The principal bundle $\coor{X}/\GL_n(\bb C)\to X$ can equivalently be seen as a $\GL_n(\bb C)$-equivariant principal bundle $\coor{X}\to\Fr(X)$. On a small enough $U_{\alpha}\in\mc{U}$ there is always a local $\GL_n(\bb C)$-equivariant  holomorphic section $s_{\alpha}: \Fr(X)|_{U_{\alpha}}\rightarrow\coor X|_{U_{\alpha}}$.  The pullback of $\omega$ along $s_{\alpha}$ yields a $\GL_n(\bb C)$-equivariant holomorphic $1$-form $\omega_{\alpha}:=s_{\alpha}^*\omega$ on $\Fr(X)|_{U_{\alpha}}$ with values in $W_n$ which satisfies the Maurer-Cartan equation. It descends to a holomorphic $W_n$-valued Maurer-Cartan $1$-form on $U_{\alpha}$, which for brevity we also denote by $\omega_{\alpha}$, and induces a holomorphic flat connection $\nabla_{\alpha}:=d+[\omega_{\alpha}, \cdot]$ on the associated trivial vector bundle 
$\mf{D}_{\alpha}:=\Fr(X)|_{U_{\alpha}}\times_{\GL_n(\bb C)}\widehat{\mc{D}}_n\rightarrow U_{\alpha}$. Finally, the pullback holomorphic bundle $\pi^*\mf{D}_{\alpha}\rightarrow T^*X|_{U_{\alpha}}$ is flat with respect to $\pi^*\nabla_{\alpha}$. We remark that the pullback of $\omega_{\alpha}\in\Omega_X^{(1, 0)}(U_{\alpha}, W_n)$ along the holomorphic projection $\pi$ induces a $W_n$-valued holomorphic $1$-form on $T^*X|_{U_{\alpha}}$ satisfying the Maurer-Cartan equation which again by abuse of notation we denote by $\omega_{\alpha}$. The direct image under $\pi$ of the sheaf of pullback horizontal sections of $\pi^*\mf{D}_{\alpha}$ with respect to $\pi^*\nabla_{\alpha}$ is isomorphic to $\coor{\pi}_*\OO_{\textrm{flat}}(\coor X\times\widehat{\mc{D}}_n)|_{U_{\alpha}}$ and hence its sections over $U_{\alpha}$ are in one-to-one correspondence with holomorphic differential operators in $\mc{D}_X(U_{\alpha})$.

Adhering to the notation in \cite{EF08} we shall  denote $k$-chains of the normalized Hochschild chain complex of an associative algebra $A$ by $(a_0, a_1, \cdots, a_k)$. Recall the nontrivial Feigin-Felder-Shoikhet normalized Hochschild $2n$-cocycle $\tau_{2n}$ of $\mc{D}_n$ with values in the dual bimodule $\mc{D}_n^*$ which naturally extends to a linear form on $\widehat{\C}_{2n}(\mc{D}_n):=\mc{D}_{n}^{\hat{\otimes}2n+1}$. It is $\GL_n(\bb C)$-basic meaning that it is invariant under the adjoint action of $\GL_n(\bb C)$ on $\mc{D}_n$ and for any $a\in\mf{gl}_n(\bb C)$,
\[\sum_{j=1}^{2n}(-1)^j\tau_{2n}(\widehat{D}_0, \cdots, \widehat{D}_{j-1}, a, \widehat{D}_j, \cdots, \widehat{D}_{2n-1})=0,\]
where $\widehat{D}_j\in\mc{D}_n$ for all $j=0, \dots, 2n-1$,
holds true.

For every intersection $U_{\alpha\beta}=U_{\alpha}\cap U_{\beta}$, let $g_{\alpha\beta}: U_{\alpha\beta}\rightarrow\GL_n(\bb C)$ be the corresponding transition map between the fiber coordinates of  $T^*X$ . The relation between two local $1$-forms $\omega_{\alpha}$ and $\omega_{\beta}$ defined in coordinates of $T^*X|_{U_{\alpha}}$ and $T^*X|_{U_{\beta}}$ is given by 
\[\omega_{\beta}=\ad(g_{\alpha\beta}^{-1})\omega_{\alpha}+g_{\alpha\beta}^{-1}dg_{\alpha\beta}.\] 
Furthermore, every flat section $\widehat{D}_{\alpha}$ of $\mf{D}_{\alpha}$ over $U_{\alpha}$ representing a unique differential operator $D_{\alpha}$ in $\mc{D}_{X}(U_{\alpha})$ transforms under a change of trivialization according to 
\[\widehat{D}_{\beta}=\ad(g_{\alpha\beta}^{-1})\widehat{D}_{\alpha}.\] 
Now, consider on every open set $U_{\alpha}$ the following  composition of morphisms
\begin{align}
\label{composofmapsdeftrdm}
&\C_{p}(\mc{D}_X(U_{\alpha}))\rightarrow\C_{p}(\Gamma_{\textrm{flat}}(\pi^{-1}(U_{\alpha}), \pi^*\mf{D}_{\alpha}))\xhookrightarrow{f}\C_{p}\big((\Omega_{T^*X}^{\bullet}(\pi^{-1}(U_{\alpha}), \pi^*\mf{D}_{\alpha}), \tilde{\nabla}_{\alpha})\big)\nonumber\\
&\xrightarrow{g}\C_{p}\big((\Omega_{T^*X}^{\bullet}(\pi^{-1}(U_{\alpha}), \widehat{\mc{D}}_{n}), d_{\dR}+[\omega_{\alpha}, \cdot])\big)\rightarrow\C_{p}\big((\Omega_{T^*X}^{\bullet}(\pi^{-1}(U_{\alpha}), \widehat{\mc{D}}_{n}), d_{\dR})\big)\nonumber\\
&\rightarrow(\Omega_{T^*X}^{2n-p}(\pi^{-1}(U_{\alpha})), (-1)^{2n-p}d_{\dR})\rightarrow(\Omega_{T^*X}^{2n-p}(\pi^{-1}(U_{\alpha})), d_{\dR}),
\end{align}
where $\tilde{\nabla}_{\alpha}$ is the covariant derivative induced by $\nabla_{\alpha}$, with corresponding connecting morphisms given by 
\begin{align*}
&(D_0, \dots, D_p)\xmapsto{\widehat{\cdot}}(\widehat{D}_0, \dots, \widehat{D}_p)\xmapsto{g\circ f}(\widehat{D}_0, \dots, \widehat{D}_p)\mapsto\sum_{k\geq0}(-1)^k(\widehat{D}_0, \dots, \widehat{D}_p)\times(\omega_{\alpha})^k\mapsto\\
&\tau_{2n}\big(\sum_{k\geq0}(-1)^k(\widehat{D}_0, \dots, \widehat{D}_p)\times(\omega_{\alpha})^k\big)\xmapsto{(-1)^{\floor{\frac{2n-p}{2}}}}\sum_{k\geq0}(-1)^{\floor{\frac{3k}{2}}}\tau_{2n}\big((\widehat{D}_0, \dots, \widehat{D}_p)\times(\omega_{\alpha})^k\big).\\
\end{align*}
Each of the above morphism are discussed in a greater detail in Section 3 in \cite{Ram11}
This composition gives a map between complexes of sheaves over $U_{\alpha}$ 
\begin{align*}
\chi_{\alpha}: \mc{C}_{\bullet}(\mc{D}_X, \mc{D}_X)|_{U_{\alpha}}&\longrightarrow\pi_*\Omega_{T^*X|_{U_{\alpha}}}^{2n-\bullet}\\
(D_0, \cdots, D_n)&\mapsto\sum_{k\geq0}(-1)^{\floor{\frac{3k}{2}}}\tau_{2n}\big((\widehat{D}_0, \cdots, \widehat{D}_n)\times(\omega_{\alpha})^k\big)
\end{align*}
where $(\omega_{\alpha})^k$ denotes the normalized Hochschild $k$-chain $(1, \omega_{\alpha}, \cdots, \omega_{\alpha})$ ($k$ copies of $\omega_{\alpha}$). 
We now want to show that the  map $\chi_{\alpha}$ extends to a map $\chi$ over all of $X$. To that aim, consider the following easy to prove binomial formula for the shuffle product  
\begin{align*}
&(a+b)^k=\sum_{j=0}^k\binom{k}{j}(a)^{k-j}\times(b)^j
\end{align*}
which we invoke in the ensuing computation. After a change of trivialization, the map $\chi_{\alpha}$ changes as
\begin{align}
\label{gluingtracedensities}
&\chi_{\beta}(D_{0}, \cdots, D_{p})=\sum_{k\geq0}(-1)^{\floor*{\frac{3k}{2}}}\tau_{2n}\Big((\widehat{D}_{0\beta}, \cdots, \widehat{D}_{p\beta})\times(\omega_{\beta})^k\Big)\nonumber\\
&=\sum_{k\geq0}(-1)^{\floor*{\frac{3k}{2}}}\tau_{2n}\Big((\ad(g_{\alpha\beta}^{-1})\widehat{D}_{0\alpha}, \cdots, \ad(g_{\alpha\beta}^{-1})\widehat{D}_{p\alpha})\times(\ad(g_{\alpha\beta}^{-1})\omega_{\alpha}-g_{\alpha\beta}^{-1}dg_{\alpha\beta})^k\Big)\nonumber\\
&=(-1)^{\floor*{3n-\frac{3p}{2}}}\tau_{2n}\Big((\ad(g_{\alpha\beta}^{-1})\widehat{D}_{0\alpha}, \cdots, \ad(g_{\alpha\beta}^{-1})\widehat{D}_{p\alpha})\times(\ad(g_{\alpha\beta}^{-1})\omega_{\alpha}-g_{\alpha\beta}^{-1}dg_{\alpha\beta})^{2n-p}\Big)\nonumber\\
&=(-1)^{\floor*{3n-\frac{3p}{2}}}\tau_{2n}\Big((\ad(g_{\alpha\beta}^{-1})\widehat{D}_{0\alpha}, \cdots, \ad(g_{\alpha\beta}^{-1})\widehat{D}_{p\alpha})\times\nonumber\\
&\quad\qquad\qquad\qquad\qquad\qquad\sum_{j=0}^{2n-p}\binom{2n-p}{j}(-1)^j(\ad(g_{\alpha\beta}^{-1})\omega_{\alpha})^{2n-p-j}\times(g_{\alpha\beta}^{-1}dg_{\alpha\beta})^{j}\Big)\nonumber\\
&=(-1)^{\floor*{3n-\frac{3p}{2}}}\tau_{2n}\Big((\ad(g_{\alpha\beta}^{-1})\widehat{D}_{0\alpha}, \cdots, \ad(g_{\alpha\beta}^{-1})\widehat{D}_{p\alpha})\times(\ad(g_{\alpha\beta}^{-1})\omega_{\alpha})^{2n-p}\Big)+\nonumber\\
&+\sum_{j=1}^{2n-p}(-1)^{j+\floor*{3n-\frac{3p}{2}}}\binom{2n-p}{j}\tau_{2n}\Big((\ad(g_{\alpha\beta}^{-1})\widehat{D}_{0\alpha}, \cdots, \ad(g_{\alpha\beta}^{-1})\widehat{D}_{p\alpha})\times\nonumber\\
&\quad\qquad\qquad\qquad\qquad\qquad\qquad(\ad(g_{\alpha\beta}^{-1})\omega_{\alpha})^{2n-p-j}\times(g_{\alpha\beta}^{-1}dg_{\alpha\beta})^{j-1}\times(g_{\alpha\beta}^{-1}dg_{\alpha\beta})\Big)\nonumber\\
&=(-1)^{\floor*{3n-\frac{3p}{2}}}\tau_{2n}\Big((\widehat{D}_{0\alpha}, \cdots, \widehat{D}_{p\alpha})\times(\omega_{\alpha})^{2n-p}\Big)\nonumber\\
&=\chi_{\alpha}(D_{0}, \cdots, D_{p})
\end{align} 
where in the second to the last line we made use of the fact that the Feigin-Felder-Schoikhet $2n$-cocycle $\tau_{2n}$ is $\GL_n(\bb C)$-basic. This demonstrates that the map is independent on the choice of trivialization and hence it is well-defined as a morphism of complexes of sheaves on $\mf{U}$. Consequently, the family of maps $(U_{\alpha}, \chi_{\alpha})$ extends to a morphism of cochain complexes of sheaves 
\begin{equation}
\label{holomorphictracedensity}
\chi: \mc{C}_{\bullet}(\mc{D}_X)\rightarrow\pi_*\Omega_{T^*X}^{2n-\bullet}
\end{equation}
on the whole of $X$ if we regard the Hochschild chain complex $\mc{C}_{\bullet}(\mc{D}_X)$ as a cochain complex by inverting degrees, i.e. $\mc{C}^{-\bullet}(\mc{D}_X, \mc{D}_X)=\mc{C}_{\bullet}(\mc{D}_X)$. For any choice of local holomorphic sections $\{s_{\alpha}\}$ on the different members of $\mc{U}$ the  corresponding induced cochain morphisms \eqref{holomorphictracedensity} are homotopic. Therefore, at the level of cohomology sheaves the morphism $\chi_*$ is canonic. Moreover, we can prove the ensuing result.  
\begin{proposition}
The morphism \eqref{holomorphictracedensity} is a quasi-isomorphism. 
\end{proposition}
\begin{proof}
The stalk of the cohomology sheaf ${\bf H}^{\bullet}(\pi_*\Omega_{T^*X}^{2n-\bullet})$ at $x\in X$ by definition is
\begin{align*}
\label{stalkofcsholderham}
{\bf H}^{-\bullet}(\pi_*\Omega_{T^*X}^{2n-\bullet})_x&\cong\h^{-\bullet}(\pi_*\Omega_{T^*X, \pi(x, v)}^{2n-\bullet})\\
&\cong\varinjlim_{\pi^{-1}(U)\ni(x, v)}\h^{-\bullet}(\Omega_{T^*X}^{2n-\bullet}(\pi^{-1}(U)))\\
&\cong\varinjlim_{X\supseteq U\ni x}\h^{-\bullet}(\Omega_X^{2n-\bullet}(U))\\
&\cong\h^{-\bullet}(\Omega_{X, x}^{2n-\bullet})\\
&\cong\h^{-\bullet}(\mc{A}_{X, \bb C, x}^{2n-\bullet})\\
&=\begin{cases}
\bb C, \textrm{if $\bullet=2n$}\\
0, \textrm{otherwise}
\end{cases}
\end{align*} 
where in the second and the fourth line we use the commutativity of direct limits with the cohomology functor, in the third line we invoke the homotopy equivalence of the holomorphic cotangent bundle and the base space and in the second to the last line we invoke the fact that $\mc{A}_{X, \bb C}^{\bullet}$ and $\Omega_X^{\bullet}$ are both (injective) resolutions of the constant sheaf $\bb C_X$. Consider the induced map on the cohomology sheaves
\[\chi_*: \mc{HH}^{-\bullet}(\mc{D}_{X}, \mc{D}_{X})\rightarrow{\bf H}^{-\bullet}(\pi_*\Omega_{T^*X}^{2n-\bullet}).\]
This map is an isomorphism if and only if it is an isomorphism on stalks. Let $(x_1, \dots, x_n)$ be the local coordinates at a point $x\in X$. From \cite[Theorem $1$]{Wod87}, it follows that for $x\in X$, 
\begin{align}
\label{stalkofhsofdiffop}
\mc{HH}^{-\bullet}(\mc{D}_X, \mc{D}_X)_x&\cong\h^{-\bullet}(\mc{C}^{-\bullet}(\mc{D}_{X, x}, \mc{D}_{X, x}))\nonumber\\
&\cong\hh_{\bullet}(\mc{D}_{X, x})\nonumber\\
&\cong\varinjlim_{X\supseteq U\ni x}\hh_{\bullet}(\mc{D}_X(U))\nonumber\\
&\cong\begin{cases}
\varinjlim_{X\supseteq U\ni x}\bb C\cdot\{c_X(U)\},\quad \textrm{if $\bullet=2n$}\nonumber\\
0,\qquad\textrm{otherwise}
\end{cases}\nonumber\\
&\cong\begin{cases}
\bb C\cdot\{c_{X, x}\},\quad \textrm{if $\bullet=2n$}\\
0,\qquad \textrm{otherwise} 
\end{cases}
\end{align}
where $c_X(U):=\sum_{\pi\in S_{2n}}\sgn(\pi)(1, u_{\pi(1)}, \dots, u_{\pi(2n)})$ with $u_{2j-1}=\partial_{x_j}, u_{2j}=x_j$ is the generator of $\hh_{2n}(D_X(U))$. By the normalization property of $\tau_{2n}$, for every $x\in X$, we have on the generator $c_{X, x}$ of $\hh_{2n}(\mc{D}_{X, x})$ 
\[\chi_{*, x}([c_{X, x}])=[1]\] 
which implies that the induced morphism
\[\chi_{*, x}: \mc{HH}^{-\bullet}(\mc{D}_{X}, \mc{D}_{X})_x\rightarrow{\bf H}^{-\bullet}(\pi_*\Omega_{T^*X}^{2n-\bullet})_x\]
is a non-trivial map between one-dimensional stalks of cohomology sheaves.  Hence, $\chi_{*, x}$ is an isomorphism. Ergo, $\chi_*$ is an isomorphism of cohomology  sheaves which concludes the proof. 
\end{proof}
\subsubsection{Generalized trace density morphism for the skew group ring $\mc{D}_X\rtimes G$} 
\label{gentdm}
Sheaves on the quotient  $X/G$ such as $\mc{D}_X\rtimes G$ can equivalently be viewed as sheaves defined in the $G$-equivariant topology of $X$. To simplify computations, we shall work on the basis $\mf{B}_X^G$ for the $G$-equivariant topology of $X$ defined in Section \ref{formalgeom}.  Here, we are interested in generalized trace density maps for strata associated to cyclic  subgroups because, as it is shown in Theorem \ref{quasiisomthm}, they are the building blocks of a quasi-isomorphism for the Hochschild chain and cochain complexes of $\mc{D}_X\rtimes G$ which we later use to compute the space of infinitesimal deformations of $\mc{D}_X\rtimes G$. However, the construction in the following goes through for any parabolic subgroup $H$ of $G$.  

Let $\langle g\rangle$ be the cyclic group generated by a group element $g$ in $G$ and let $X_{\langle g\rangle}^i$ be a stratum of codimension $l_g^i$ with closure lying in the corresponding connected component $X_i^g$ of the $g$-fixed point submanifold. Let $\pi_i^g: T^*X_i^g\rightarrow X_i^g$ be the cotangent bundle to  $X_i^g$. We first extend map \eqref{collapsingmap}, 
\begin{align}
\label{restrictedtrdm}
p: \mc{D}_X\rtimes G&\rightarrow j_{\langle g\rangle*}^i\pi^{\textrm{coor}}_*\OO_{\textrm{flat}}(\mc{N}^{\textrm{coor}}|_{X_{\langle g\rangle}^i}\times\mc{A}_{n-l_g^i, l_g^i}^{\langle g\rangle}),
\end{align}  
to the fixed point submanifold $X_i^g$. 
\begin{lemma}
The map \eqref{restrictedtrdm} uniquely extends to a morpism of sheaves 
\begin{align*}
 \bar{p}: \mc{D}_X\rtimes G&\rightarrow j_{i*}^g\pi^{\textrm{coor}}_*\OO_{\textrm{flat}}(\mc{N}^{\textrm{coor}}\times\mc{A}_{n-l_g^i, l_g^i}^{\langle g\rangle}).
 \end{align*}
 on $X_i^g$. 
\end{lemma}
\begin{proof}
We distinguish two cases: ${\bf {1}})$ the codimension of the stratum $X_{\langle g\rangle}^i$ is  equal to or bigger than $1$, ${\bf {2}})$ the stratum $X_{\langle g\rangle}^i$ is the prinicipal (dense and open) stratum $\mathring{X}$ in $X$.

${\bf {1}})$ 
Let $W_x$ be a Stein $K$-invariant linear slice centered at a point $x$ on a stratum $X_K^j$ contained in $X_i^g$ as above. By Hartog's Theorem, the third arrow in \eqref{compofres} is bijective. Hence, the image  
of morphism \eqref{compofres}, when $H=\langle g\rangle$, coincides with the image of 
\begin{align*}
&\mc{D}_X(W_x\setminus X_K^j)\rtimes\langle g\rangle\xrightarrow{\cong}\invlim_{\substack{W_y\\ y\in X_L,~ L\leq\langle g\rangle}}\mc{D}_X(\ind_L^{\langle g\rangle}W_y) \rtimes\langle g\rangle\nonumber\\
&\twoheadrightarrow\prod_{\substack{W_y\\ \textrm{ $y\in X_{\langle g\rangle}^i$}}}\{(s)\in\mc{D}_X(W_y)\rtimes\langle g\rangle:~\res_{U_{\beta'}}^{U_{\beta}}(s)=s'\}\rightarrow\coor{\pi}_*\OO_{\textrm{flat}}(\coor{\mc{N}}|_{X_{\langle g\rangle}^i}\times\mc{A}_{n-l_g^i, l_g^i}^{\langle g\rangle})(W_{x, \langle g\rangle}^i).
\end{align*}
Here, $W_y$ is a Stein $L$-invariant slice. The preimage $\hat{s}_1$ in $\mc{D}_X(W_x\setminus X_K^j)\times\langle g\rangle$ of every section $\hat{s}$ in the image of \eqref{compofres} is non-empty. Furthermore, as the codimension of $X_K^j$ is at least $2$ in $X$,  by  Hartog's Theorem, it follows $\mc{D}_X(W_x\setminus X_K^j)\rtimes\langle g\rangle\cong\mc{D}_X(W_x)\rtimes\langle g\rangle$. By the construction in Section \ref{constrofdxg}, $\hat{s}_1$ determines a unique section $\hat{s}_2$ of $\coor{\pi}_*\OO_{\textrm{flat}}(\coor{\mc{N}}\times\mc{A}_{n-l_g^i, l_g^i}^{\langle g\rangle})(W_x\cap X_i^g)$ with $\hat{s}_2|_{W_{x, \langle g\rangle}^i}=\hat{s}$. Since all sections $\hat{s}_2$ coincide on the open subset $W_{x, \langle g\rangle}^i$ of $W_x\cap X_i^g$, by the identity theorem, they are identical on $W_x\cap X_i^g$. This determines a well-defined extension $\bar p$ of $p$ on $X_i^g$.

${\bf {2}})$ Assume $X_{\langle g\rangle}^i=\mathring{X}$. Assume $\codim(X_K^j)=1$. Here, we only consider the case of a basic open sets $\ind_K^GW_x$ with $x$, centered on $X_K^j$.  Each element $D\cdot k$ in $\mc{D}_{X}(W_x)\rtimes K$ is  uniquely represented by a pair of sections $u=(t|_{W_x\setminus X_K^j}\cdot k, s|_{W_{x, K}^j})$ of $\coor{\pi}_*\OO_{\textrm{flat}}(\coor{\mathring{X}}\times\widehat{\mc{D}}_n)(W_x\setminus X_K^j)\rtimes K$ and $\coor{\pi}_*\OO_{\textrm{flat}}(\coor{\mc{N}}\times\mc{A}_{n-1, 1}^K)(W_{x, K}^j)$ satisfying the gluing conditions from Section \ref{constrofdxg}. If we set $Y=\mathring{X}\sqcup X_K^j$, the same element $D\cdot k$ uniquelly determines a flat section $r\cdot k$ in $\coor{\pi}_*\OO_{\textrm{flat}}(\coor{Y}\times\mc{D}_n)(U_{\alpha})\rtimes K$. Clearly, the restriction maps satisfy $\res_{W_x\setminus X_K^j}^{W_x}(r\cdot k)=\res_{W_x\setminus X_K^j}^{W_x}(u)=t|_{W_x\setminus X_K^j}\cdot k$. Hence, $r$ is the unique flat section in $\coor{\pi}_*\OO_{\textrm{flat}}(\coor{Y}\times\mc{D}_n)(W_x)$ which is assigned to $D\cdot k$ and extends $t|_{W_x\setminus X_K^j}$. This gives the wanted extension $\bar p$. We further extend \eqref{restrictedtrdm} step by step to the union of $\mathring{X}$ with all strata of codimension $1$. 
\end{proof}
Let in the following $\mc{U}=\{U_{\alpha}\}$ be a family of Stein $H$-invariant slices at points $x$ on $X_i^g$ with stabilizers $H\supseteq\langle g\rangle$ which trivialize  $TX|_{X_i^g}$ such that the sets $W_{\alpha}=U_{\alpha}\cap X_i^g$ form an open cover of $X_i^g$. We pullback $\omega$ along a locally holomorphic section $s_{\alpha}$ of $\coor{\mc{N}}$ over $W_{\alpha}=U_{\alpha}\cap X_i^g$ to a $\mc{A}_{n-l_g^i, l_g^i}^{\langle g\rangle}$-valued  holomorphic $1$-form $\omega_{\alpha}:=s_{\alpha}^*\omega$ on $W_{\alpha}$ satisfying the Maurer-Cartan equation in a similar fashion to Section \ref{sectdoodo}. A further  pullback of $\omega_{\alpha}$ along $\pi_i^g$ yields a holomorphic Maurer-Cartan form on $T^*X_i^g|_{W_{\alpha}}$ with values in $\mc{A}_{n-l_g^i, l_g^i}^{\langle g\rangle}$, as desired. For each linearly independent trace $\phi$ of $\widehat{\mc{D}}_{l_g^i}\rtimes\langle g\rangle$ we define a $\GL_{n-l_g^i}(\bb C)\times Z$-basic normalized $2n-2l_g^i$ Hochschild cocycle $\psi_{2n-2l_g^i}$ of $\mc{A}_{n-l_g^i, l_g^i}^{\langle g\rangle}$ with values in the dual left $(\mc{A}_{n-l_g^i, l_g^i}^{\langle g\rangle})^e$-module $\mc{A}_{n-l_g^i, l_g^i}^{\langle g\rangle*}$by 
\[\psi_{2n-2l_g^i}(a_1\otimes b_1, \dots, a_{2n-2l_g^i}\otimes b_{2n-2l_g^i}):=\tau_{2n-2l_g^i}(a_1, \dots, a_{2n-2l_g^i})\phi(b_1\cdot\dots\cdot b_{2n-2l_g^i}).\]
This cocycle extends to a linear form on  $\C_{2n-2l_g^i+1}(\mc{A}_{n-l_g^i, l_g^i}^{\langle g\rangle})$, which we again denote by $\psi_{2n-2l_g^i}$. With the help of it, we define locally on basic sets $\ind_{H}^GU_{\alpha}$ an identical to  \eqref{composofmapsdeftrdm} composition of maps which yields the following local form for the generalized trace density morphism for $\mc{D}_X\rtimes G$ in the $G$-equivariant topology
\begin{align}
\label{locgtd}
\chi_i^g:\mc{C}_{\bullet}(\mc{D}_X\rtimes G)|_{\ind_{H}^GU_{\alpha}}&\rightarrow\Big( j_{i*}^g\pi_{i*}^g\Omega_{T^*X_i^g}^{2n-2l_g^i-\bullet}\Big)^G|_{\ind_{H}^GU_{\alpha}}\nonumber\\
(D_0g_0, \dots, D_pg_p)&\mapsto \sum_{k\geq0}(-1)^{\floor{\frac{3k}{2}}}\psi_{2n-2l_g^i}((s_0\circ\pi_{\langle g\rangle}^i, \dots, s_p\circ\pi_{\langle g\rangle}^i)\times(\omega_{\alpha})^k)
\end{align}
where $s_0, \dots, s_p$ are the horizontal sections in $j_{i*}^g\coor{\pi}_*\OO_{\textrm{flat}}(\coor{\mc{N}}\times\mc{A}_{n-l_g^i, l_g^i}^{\langle g\rangle})$ corresponding to $D_0g_0, \dots, D_pg_p$. A lengthy calculation similar to \eqref{gluingtracedensities} shows that the map  \eqref{locgtd} is independent on the choice of $\ind_H^GU_{\alpha}$ by the $\GL_{n-l_g^i}(\bb C)\times Z$-basicness of $\psi_{2n-2l_g^i}$. Hence, it extends to a well-defined morphism of cochain complexes 
\begin{align*}
\chi_i^g:\mc{C}_{\bullet}(\mc{D}_X\rtimes G)&\rightarrow\Big( j_{i*}^g\pi_{i*}^g\Omega_{T^*X_i^g}^{2n-2l_g^i-\bullet}\Big)^G
\end{align*}
on the $G$-equivariant topology of $X$ and $X/G$, respectively. This map is the desired generalized trace density morphism of the Hochschild chain complex of $\mc{D}_X\rtimes G$. 
 \begin{theorem}
 \label{quasiisomthm}
For every choice of a linearly independent trace $\phi_{\langle  g\rangle}^i$ of $\widehat{\mc{D}}_{l_g^i}\rtimes\langle g\rangle$ and for every family of holomorphic  Maurer-Cartan forms $\{\omega_{\alpha}\}$ on the cotangent bundle $T^*X_i^g$ of $X_i^g$ with values in $\mc{A}_{n-l_g^i, l_g^i}^{\langle g\rangle}$ the map
\begin{align}
\label{trdmofdsmashg}
\oplus_{i, g\in G} \chi_i^g: \mc{C}_{\bullet}(\mc{D}_X\rtimes G)\rightarrow\Big(\oplus_{i,g\in G}j_{i*}^g\pi_{i*}^g\Omega_{T^*X_i^g}^{2n-2l_g^i-\bullet}\Big)^G
\end{align} 
is a quasi-isomorphism. Moreover, the induced morphism at the level of homology sheaves is canonical.       
 \end{theorem}
 \begin{proof}
 The zeroth Hochschild cohomology group of $\mc{D}_{l_g^i}\rtimes\langle g\rangle$ with values in the dual bimodule $(\mc{D}_{l_g^i}\rtimes\langle g\rangle)^{*}$ defines the space of traces on $\mc{D}_{l_g^i}\rtimes\langle g\rangle$. From \cite[Proposition 3.1]{AFLS00} we infer the isomorphism
 \begin{equation*}
 \mf{r}: \hh^{0}(\mc{D}_{l_g^i}\rtimes\langle g\rangle, \mc{D}_{l_g^i}\rtimes\langle g\rangle^*)\cong\Big(\oplus_{k=1}^{\ord(g)}\hh^0(\mc{D}_{l_g^i}, \mc{D}_{l_g^i}g^k*)^{\langle g\rangle}\Big) 
 \end{equation*}
 where $\ord(g)$ is the order of the generator $g$ in $G$. As each group $\hh^0(\mc{D}_{l_g^i}, \mc{D}_{l_g^i}g^k*)^{\langle g\rangle}$ is spanned by the $g^k$-twisted trace $\tr_{g^k}$, defined in \cite{Fed00}, we have for the image of every trace $\phi_{\langle g\rangle}^i$ of $\mc{D}_{l_g^i}\rtimes\langle g\rangle$, associated to the normal bundle to $X_i^g$, under $\mf{r}$
 \begin{equation*}
\mf{r}(\phi_{\langle g\rangle}^i)=\sum_{k=1}^{\ord(g)}\lambda_k\tr_{g^k}(\cdot)
 \end{equation*}
where for each $k$, the constant $\lambda_k$ is a complex number,   including possibly zero. On the other hand, invoking the fact that $\mc{D}_X\rtimes G$ and $\mc{D}_X$ are sheaves of Calabi-Yau algebras of dimension $2n$ we  have the natural identification
 \begin{align*}
 \mf{s}:\mc{HH}_{\bullet}(\mc{D}_X\rtimes G)&\cong\Big(\oplus_{g\in G}\mc{HH}_{\bullet}(\mc{D}_X, \mc{D}_Xg)\Big)^{G}
 \end{align*}
which combined with the induced map of \eqref{trdmofdsmashg} on homology yields the map
 \begin{align}
 \label{aidinginducedmap}
\oplus_{g, i}\chi_{i*}^g\circ\mf{s}^{-1}:~\Big(\oplus_{g\in G}\mc{HH}_{\bullet}(\mc{D}_X, \mc{D}_Xg)\Big)^{G}
&\rightarrow\Big(\oplus_{i,g\in G}{\bf H}^{-\bullet}\big(j_{i*}^g\pi_{i*}^g\Omega_{T^*X_i^g}^{2n-2l_g^i-\bullet}\big)\Big)^G
 \end{align}
 where 
 \[\chi_{i*}^g\circ\mf{s}^{-1}(D_0g, D_1, \dots, D_p)=\sum_{i_0, \dots, i_p}\sum_{k}(-1)^{\floor{\frac{3k}{2}}}\tilde{\psi}_{2n-2l_g^i}(\widehat{D}_{0, i_0}\otimes \widehat{D}_{0, i_0}^{\perp}g, \dots, \widehat{D}_{p, i_p}\otimes \widehat{D}_{p, i_p}^{\perp})\] 
 with 
 \[\tilde{\psi}_{2n-2l_g^i}(a_0\otimes b_0, \dots, a_{2n-2l_g^i}\otimes b_{2n-2l_g^i})=\tau_{2n-2l_g^i}(a_0, \dots, a_{2n-2l_g^i})\mf{r}(\phi_{\langle g\rangle}^i)(b_0\cdots b_{2n-2l_g^i}),\] 
 $a_i\otimes b_i\in\mc{A}_{n-l_g^i, l_g^i}^{\langle g\rangle}$. If we denote by $q$ the projection of $X$ onto the orbifold quotient $X/G$, then every point $\bar x$ on $X/G$ is the image $\bar x=q(x)$ of an element $x$ on $X$ with stabilizer some subgroup $H$ of $G$. Consequently, we get for the stalk of $\mc{D}_X\rtimes G$ at $\bar x$ 
 \begin{align*}
 (\mc{D}_X\rtimes G)_{\bar x}=\colim_{q^{-1}(\bar V)\ni q^{-1}(\bar x)}\mc{D}_X(q^{-1}(\bar V))\rtimes G=\mc{D}_{X|\ind_H^G\{x\}}\rtimes G
 \end{align*}
 where $\ind_H^G\{x\}:=\sqcup_{l\in G/H} l\{x\}$ is the closed set of $G$-translates of the single point $x$ in $X$. Morphism \eqref{trdmofdsmashg} is a quasi-isomorphism if and only if for every $\bar x$ in $X/G$, the stalk of \eqref{aidinginducedmap} is an isomorphism. For the left hand side of the  induced map on the homology we get
\begin{align*}
\mc{HH}_{\bullet}(\mc{D}_X\rtimes G)_{\bar x}&\cong\hh_{\bullet}((\mc{D}_X\rtimes G)_{\bar x})\\
&\cong\Big(\oplus_{g\in G}\hh_{\bullet}(\mc{D}_{X|\ind_H^G\{x\}}, \mc{D}_{X|\ind_H^G\{x\}}g)\Big)^G\\
&\cong\Big(\oplus_{l\in G/H}\oplus_{g\in lHl^{-1}}\hh_{\bullet}(\mc{D}_{X,lx}, \mc{D}_{X,lx}g)\Big)^G\\
&\cong\Big(\bb CG\otimes_{\bb CH}\oplus_{g\in H}\hh_{\bullet}(\mc{D}_{X, x}, \mc{D}_{X, x}g)\Big)^G\\
&\cong\Big(\oplus_{g\in H}\hh_{\bullet}(\mc{D}_{X, x}, \mc{D}_{X, x}g)\Big)^H\\
&=\oplus_{C_H(g)\in\Conj(H)}\Big(\hh_{\bullet}(\mc{D}_{X, x}, \mc{D}_{X, x}g)\Big)^{Z_H(g)}
 \end{align*}
where $\Conj(H)$ is the set of conjugacy classes in $H$, $C_H(g)$ is the conjugacy class of $g$ in $H$ and $Z_H(g)$ is the centralizer group of $g$ in $H$, respectively. Similarly, we get for the stalk of the homology sheaf at $\bar x$ on the right hand side of  \eqref{trdmofdsmashg} 
\begin{align*}
\Big(\oplus_{i,g\in G}{\bf H}^{-\bullet}\big(j_{i*}^g\pi_{i*}^g\Omega_{T^*X_i^g}^{2n-2l_g^i-\bullet}\big)\Big)_{\bar x}^G&\cong\oplus_{C_H(g)\in\Conj(H)}{\bf H}^{-\bullet}(j_{i*}^g\pi_{i*}^g\Omega_{T^*X_i^g, x}^{2n-2l_g^i-\bullet})^{Z_H(g)}\\
&\cong\oplus_{C_H(g)\in\Conj(H)}{\bf H}^{2n-2l_g^i-\bullet}(\pi_{i*}^g\Omega_{T^*X_i^g, x}^{-\bullet})^{Z_H(g)}\\
&\cong\oplus_{C_H(g)\in\Conj(H)}{\bf H}^{2n-2l_g^i-\bullet}(\Omega_{X_i^g, x}^{-\bullet})^{Z_H(g)}\\
&\cong\oplus_{C_H(g)\in\Conj(H)}{\bf H}^{2n-2l_g^i-\bullet}(\mc{A}_{X_i^g, \bb C, x}^{-\bullet})^{Z_H(g)}\\
\end{align*}
where we dropped the index $i$, because a point cannot  simultaneously be on more than one connected component of a fixed point submanifold, and in the third line we used the holomorphic Poincare's Lemma and the last isomorphism follows from the fact that $\mc{A}_{X, \bb C}^{\bullet}$ and $\Omega_X^{\bullet}$ are both resolutions of the constant sheaf $\bb C_X$. We know from \cite{FT10} that for every $g$ in $G$, the cohomology group $\hh_{\bullet}(\mc{D}_{X, x}, \mc{D}_{X, x}g)$ is spanned by the cocycle 
\[c_{2n-2l_g^i, x}:=\sum_{\sigma\in S_{2n-2l_g^i}}1\otimes u_{\sigma(1)}\otimes\dots\otimes u_{\sigma(2n-2l_g^i)}\] with $u_{2j-1}=\partial_{x_j}$ and $u_{2j}=x_j$. The normalisation property of the Feigin-Felder-Shoikhet cocycle implies 
\begin{align*}
\chi_{\langle g\rangle, x}^i(\mf{s}^{-1}(c_{2n-2l_g^i, x}))&=(-1)^{\floor{3(n-l_g^i)}}\psi_{2n-2l_g^i}(\mf{s}^{-1}(c_{2n-2l_g^i, x}))\\
&=(-1)^{\floor{3(n-l_g^i)}}\sum_{k=1}^{\ord(g)}\lambda_k\tau_{2n-2l_g^i}(c_{2n-2l_g^i, x})\tr_{g^k}(1)\\
&=(-1)^{\floor{3(n-l_g^i)}}\sum_{k=1}^{\ord(g)}\lambda_k\tr(g^k)\neq0
\end{align*}
which means that for each $g$, the composition $\chi_{\langle g\rangle, x*}^i\circ\mf{s}_x^{-1}$ is a non-zero map from the generator of the $1$-dimensional homology group 
$\hh_{\bullet}(\mc{D}_{X, x}, \mc{D}_{X, x}g)^{Z_H(g)}$ to the $1$-dimensional homology group ${\bf H}^{2n-2l_g^i-\bullet}(\mc{A}_{X_{\langle g\rangle}, \bb C, x}^{-\bullet})^{Z_H(g)}$. Hence, it is an isomorphism. Thus, $\chi_{\langle g\rangle, x*}^i$ and correspondingly the direct sum $\oplus_{i, g\in G}\chi_{\langle g\rangle, x*}^i$ are invertible maps. Ergo, Morphism  \eqref{trdmofdsmashg} is a quasi-isomorphism. Moreover, by normalization the stalk of the map \eqref{trdmofdsmashg} can be made independant on the particular choice of a trace $\phi_{\langle g\rangle}^i$. Ergo,  Morphism  \eqref{trdmofdsmashg} induces a canonical isomorphism at the level of homology.  
\end{proof}
Composing the quasi-isomorphisms from Corollary \ref{calabiyauqiso} and \eqref{trdmofdsmashg} one obtains a new quasi-isomorphism which can be interpreted as a generalized trace density morphism for the Hochschild cochian complex of $\mc{D}_X\rtimes G$. To avoid repetition we leave the standard proof of the next result to the reader.
\begin{corollary}
\label{cohomologicaltrdcor}
There is a quasi-isomorphism 
\begin{align}
\label{cohomologicaltrd}
\mc{X}:\mc{C}^{\bullet}(\mc{D}_X\rtimes G, \mc{D}_X\rtimes G)\rightarrow\Big(\oplus_{i,g\in G}j_{i*}^g\pi_{i*}^g\Omega_{T^*X_i^g}^{\bullet-2l_g^i}\Big)^G.
\end{align}  
The induced morphism at the level of cohomology sheaves is canonical.
\end{corollary}
\subsection{The space of filtered infinitesimal deformations of $\mc{D}_X\rtimes G$}
\label{spfiltinfdef}
In the case of the Calabi-Yau algebras $(\mc{D}_X)^G$ and $\mc{D}_X\rtimes G$ the inclusion of the subcomplex of filtration preserving Hochschild cochains into the complex of all Hochschild cochains is a quasi-isomorphism. We only show this for $\mc{D}_X\rtimes G$ as the proof for $(\mc{D}_X)^G$ is analogous. 
\begin{proposition}
\label{caninclqiso}
The canonical inclusion $\mf{i}:\mc{C}_f^{\bullet}(\mc{D}_X\rtimes G, \mc{D}_X\rtimes G)\hookrightarrow\mc{C}^{\bullet}(\mc{D}_X\rtimes G, \mc{D}_X\rtimes G)$ is a quasi-isomorphism.
\end{proposition}
\begin{proof}
Denote by $X^{\times k}$ the $k$-fold Cartesian product of $X$. Let $\delta_k: X\rightarrow X^{\times k}, x\mapsto(x, \dots, x)$ be the diagonal embedding of $X$ in $X^{\times k}$. Recall from the definition of the exterior tensor product that $\delta_k^*\mc{D}_{X^{\times k}}\cong\mc{D}_X^{\hat{\otimes}k}$ and $\delta_k^*\Sym^{\bullet}(\mc{T}_{X^{\times k}})=\Sym^{\bullet}(\mc{T}_X)^{\hat{\otimes}k}$. By definition the nonnegative decreasing filtration of $\mc{C}_{f}^{\bullet}(\mc{D}_X\rtimes G, \mc{D}_X\rtimes G)$ results in a first quadrant spectral sequence $E_r^{pq}$ with zeroth sheet
\begin{align*}
E_0^{pq}&:=\gr^p\mc{C}_f^{p+q}(\mc{D}_X\rtimes G, \mc{D}_X\rtimes G)\\
&\cong\{F\in\sHom_{\bb{C}}(\gr^q(\mc{D}_X\rtimes G^{\hat{\otimes}n}), \gr^{q-p}(\mc{D}_X\rtimes G))~|~\textrm{for all $q\geq p$ and $n\geq0$}\}\\
&\cong\{F\in\sHom_{\bb{C}}(\gr^q(\delta_n^*\mc{D}_{X^{\times n}})\rtimes G^{\times n}, \gr^{q-p}(\mc{D}_X)\rtimes G)~|~\textrm{for all $q\geq p$, $n\geq0$}\}\\
&\cong\{F\in\sHom_{\bb{C}}(\delta_n^*\Sym^q(\mc{T}_{X^{\times n}})\rtimes G^{\times n}, \Sym^{q-p}(\mc{T}_X)\rtimes G)~|~\textrm{for all $q\geq p$, $n\geq0$}\}\\
&\cong\{F\in\sHom_{\bb{C}}(\big[\Sym^{\bullet}(\mc{T}_X)^{\hat{\otimes}n}\big]_q\rtimes G^{\times n}, \Sym^{q-p}(\mc{T}_X)\rtimes G)~|~\textrm{for all $q\geq p$,  $n\geq0$}\}\\
&\cong\{F\in\sHom_{\bb{C}}(\big[\Sym^{\bullet}(\mc{T}_X)\rtimes G^{\hat{\otimes}n}\big]_q, \Sym^{q-p}(\mc{T}_X)\rtimes G)~|~\textrm{for all $q\geq p$, $n\geq0$}\}\\
&\cong\big[\mc{C}^{p+q}(\Sym^{\bullet}(\mc{T}_X)\rtimes G,\Sym^{\bullet}(\mc{T}_X)\rtimes G)\big]_{p}
\end{align*}
with $p+q=n$ where in the third and the fifth line we used the definition of a topologically completed tensor product, and $\big[\cdot\big]_q$ denotes the homogeneous part of (homological) degree $q$ of the respective graded algebra. To shorten in the following the notation denote by $\Delta_G$ the image of the diagonal homomorphism $\Delta: G\rightarrow G\times G$. The first sheet of the spectral sequence is accordingly becomes 
\begin{align*}
E_1^{pq}&={\bf{H}}^{p+q}(\gr^p\mc{C}_f^{\bullet}(\mc{D}_X\rtimes, \mc{D}_X\rtimes G))\\
&\cong\big[{\bf H}^{p+q}(\mc{C}^{\bullet}(\Sym^{\bullet}(\mc{T}_X)\rtimes G, \Sym^{\bullet}(\mc{T}_X)\rtimes G))\big]_p\\
&=\big[\sExt_{\Sym^{\bullet}(\mc{T}_X)^e\rtimes G\times G}^{\bullet}(\Ind^{\Sym^{\bullet}(\mc{T}_X)^e\rtimes G\times G}_{\Sym^{\bullet}(\mc{T}_X)^e\rtimes\Delta_G}\Sym^{\bullet}(\mc{T}_X), \Sym^{\bullet}(\mc{T}_X)\rtimes G)\big]_p\\
&=\big[\big(\oplus_{g\in G}\sExt_{\Sym^{\bullet}(\mc{T}_X)^e}^{\bullet}(\Sym^{\bullet}(\mc{T}_X), \Sym^{\bullet}(\mc{T}_X)\cdot g)\big)^{G}\big]_p\\
&\cong\Big(\oplus_{g\in G}\big[\mc{HH}^{p+q}(\Sym^{\bullet}(\mc{T}_X), \Sym^{\bullet}(\mc{T}_X) g)\big]_p\Big)^G\\
&\cong\Big(\oplus_{g\in G}\big[\mc{HH}_{2n-p-q}( \Sym^{\bullet}(\mc{T}_X), \Sym^{\bullet}(\mc{T}_X)\cdot g)\big]_p\Big)^G
\end{align*}
where in the fourth line we used \cite[Proposition 2.2.9]{KS94} and in the last line we used that $\Sym^{\bullet}(\mc{T}_X)$ is a sheaf of Calabi-Yau algebras with dimension $2n$ as per Proposition \ref{symcalabiyau}.We proceed as in \cite[Proposition 8]{DE05} to further simplify $E_1^{pq}$. There is a natural holomorphic splitting of the holomorphic cotangent bundle $T^*X=T^*X^g\oplus N^g$, where $N^g$ is the normal bundle to the fixed point submanifold $X^g$. With the help of the resolution \eqref{bimodkoszulresolsym} of $\Sym^{\bullet}(\mc{T}_X)$, we get for the homology sheaf of $\mc{C}_{\bullet}(\Sym^{\bullet}(\mc{T}_X), \Sym^{\bullet}(\mc{T}_X)\cdot g)$
\begin{align}
\label{twistedhhofsym}
&\mc{HH}_{\bullet}(\Sym^{\bullet}(\mc{T}_X), \Sym^{\bullet}(\mc{T}_X)\cdot g)=\sTor_{\bullet}^{~~\Sym^{\bullet}(\mc{T}_X)^e}\big(\Sym^{\bullet}(\mc{T}_X)\cdot g, \Sym^{\bullet}(\mc{T}_X)\big)\nonumber\\
&=\h_{\bullet}(\Sym^{\bullet}(\mc{T}_X)\cdot g\otimes_{\Sym^{\bullet}(\mc{T}_X)^e}\Sym^{\bullet}(\mc{T}_X)\otimes_{\bb C}\Sym^{\bullet}(\mc{T}_X)\otimes_{\delta^{-1}\OO_{Y|\Delta}}\bigwedge^{\bullet}\delta^{-1}\mc{T}_{Y|\Delta}^*)\nonumber\\
&\cong\h_{\bullet}(\Sym^{\bullet}(\mc{T}_X)\cdot g\otimes_{\delta^{-1}\OO_{Y|\Delta}}\bigwedge^{\bullet}\delta^{-1}\mc{T}_{Y|\Delta}^*)\nonumber\\ 
&=\Sym^{\bullet}(\mc{T}_{X^g})\otimes_{\delta^{-1}\OO_{Y^g|\Delta}}\bigwedge^{\bullet-2l_g}\delta^{-1}\mc{T}_{Y^g|\Delta}^*
\end{align}
where the complex in last line of \eqref{twistedhhofsym} 
is equipped with the vanishing differential and $l_g$ is the complex codimension of $X^g$ in $X$. The step from the second to the last line to the last line in \eqref{twistedhhofsym} is a sheaf theoretic version of \cite[Proposition 4]{Ann05}. 
Making use of the relation $\Omega_{Y}^1\cong\mc{T}_Y^*$ we get 
\begin{equation}
E^{pq}_1=\Big(\oplus_{g\in G}\Sym^{\bullet}(\mc{T}_{X^g})\otimes_{\delta^{-1}\OO_{Y^g|\Delta}}\delta^{-1}\bigwedge^{2n-2l_g-(p+q)}\mc{T}_{Y^g|\Delta}^*\Big)^G.
\end{equation}
As $E_r^{pq}$ is a first quadrant spectral sequence, it converges to $\mc{HH}_f^{\bullet}(\mc{D}_X\rtimes G, \mc{D}_X\rtimes G)$. On the other hand the spectral cohomological sequence $E_r^{pq}$ collapses on the $p$-axis, as shown in the picture below 
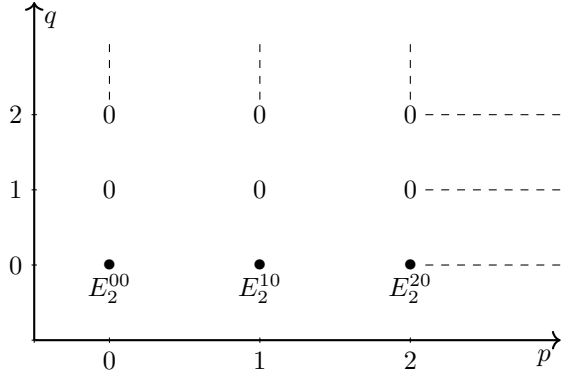
\begin{equation*}
\begin{tikzpicture}
\draw[thick,->] (0,0) -- (7,0) node[anchor=north east] {$p$};
\draw[thick,->] (0,0) -- (0,4.5) node[anchor=north west] {$q$};
\draw (0 cm,1pt) -- (0 cm,-1pt) node[anchor=north] {};
\draw (1cm, 1pt) -- (1cm, -1pt) node[anchor=north] {$0$};
\draw (3cm, 1pt) -- (3cm, -1pt) node[anchor=north] {$1$};
\draw (5cm, 1pt) -- (5cm, -1pt) node[anchor=north] {$2$};
\draw(1pt, 0 cm)--(-1pt, 0 cm) node[anchor=east] {$$};
\draw (1pt,1 cm) -- (-1pt,1 cm) node[anchor=east] {$0$};
\draw (1pt,2 cm) -- (-1pt,2 cm) node[anchor=east] {$1$};
 \draw (1pt,3 cm) -- (-1pt,3 cm) node[anchor=east] {$2$};
 \draw (1cm, 3cm) node []{$0$} ;
 \draw (3cm, 3cm) node []{$0$};
 \draw (5cm, 3cm) node []{$0$};

 \draw (1cm, 2cm) node []{$0$};
 \draw (3cm, 2cm) node []{$0$};
 \draw (5cm, 2cm) node []{$0$};

 \draw (1cm, 1cm) node []{$\bullet$} node[anchor=north] {$E_2^{00}$};
 \draw (3cm, 1cm) node []{$\bullet$} node[anchor=north] {$E_2^{10}$};
 \draw (5cm, 1cm) node []{$\bullet$} node[anchor=north] {$E_2^{20}$}; 

\draw[dashed](1cm, 3.2cm)--(1cm, 4cm) {};
\draw[dashed](3cm, 3.2cm)--(3cm, 4cm) {};
\draw[dashed](5cm, 3.2cm)--(5cm, 4cm) {};
\draw[dashed](5.2,3)--(7,3){};
\draw[dashed](5.2,2)--(7,2){};
\draw[dashed](5.2,1)--(7,1){};
\end{tikzpicture}
\end{equation*}   
Accounting for the fact that the differential $d_1^{pq}: E_1^{pq}\rightarrow E_1^{p+1, q}$ is the one of cohomological degree $-1$, obtained from the Koszul resolution of $\Sym^{\bullet}(\mc{T}_X)$, we obtain
\begin{equation}
\label{secondsheetscos}
E_2^{p0}:=\Big(\oplus_{g\in G}{\bf{H}}^{-p}\big(\Sym^{\bullet}(\mc{T}_{X^g})\otimes_{\delta^{-1}\OO_{Y^g|\Delta}}\delta^{-1}\bigwedge^{2n-2l_g-\bullet}\mc{T}_{Y^g|\Delta}^*\big)\Big)^G.
\end{equation}
Combining the convergence of $E_r^{pq}$ with  \eqref{secondsheetscos} yields 
\begin{equation}
\label{convergencescos}
E_{p0}^{\infty}=\Big(\oplus_{g\in G}{\bf{H}}^{-p}\big(\Sym^{\bullet}(\mc{T}_{X^g})\otimes_{\delta^{-1}\OO_{Y^g|\Delta}}\delta^{-1}\Omega_{Y^g|\Delta}^{2n-2l_g-\bullet}\big)\Big)^G\cong\mc{HH}_f^p(\mc{D}_X\rtimes G, \mc{D}_X\rtimes G).
\end{equation}
On the other hand, the filtration of $\mc{D}_X\rtimes G$ induces a bounded below and exhaustive filtration on the Hochschild chain complex $\mc{C}_{\bullet}(\mc{D}_X\rtimes G)$. This defines a spectral homological sequence $E_{pq}^r$ with 
\[E_{pq}^1:=\mc{HH}_{p+q}(\Sym^{\bullet}(\mc{T}_X)\rtimes G)\cong\Big(\oplus_{g\in G}\Sym^{\bullet}(\mc{T}_X^g)\otimes_{\delta^{-1}\OO_{Y^g|\Delta}}\bigwedge^{p+q-2l_g}\delta^{-1}\mc{T}_{Y^g|\Delta}^*\Big)^G.\]
According to the classical convergence theorem \cite[Theorem 5.5.1]{weibel94} this  spectral homological sequence converges to $\mc{HH}_{p+q}(\mc{D}_X\rtimes G)\cong\mc{HH}^{2n-p-q}(\mc{D}_X\rtimes G, \mc{D}_X\rtimes G)$ where the last isomorphism derives from the fact that $\mc{D}_X\rtimes G$ is a sheaf of Calabi-Yau algebras of dimension $2n$. Obviously, $E_{pq}^r$ collapses on the $p$-axis. Hence, 
\begin{equation}
\label{convergenceshs}
E_{p0}^{\infty}=\Big(\oplus_{g\in G}{\bf{H}}^{p}(\Sym^{\bullet}(\mc{T}_X^g)\otimes_{\delta^{-1}\OO_{Y^g|\Delta}}\bigwedge^{\bullet-2l_g}\delta^{-1}\mc{T}_{Y^g|\Delta}^*)\Big)^G\cong\mc{HH}^{2n-p}(\mc{D}_X\rtimes G, \mc{D}_X\rtimes G)
\end{equation}
From Isomorphism \eqref{convergencescos} and Isomorphism \eqref{convergenceshs} we infer that $\mc{HH}_f^p(\mc{D}_X\rtimes G, \mc{D}_X\rtimes G)\cong\mc{HH}^p(\mc{D}_X\rtimes G, \mc{D}_X\rtimes G)$ for every $p\leq 2n$. 
Since the map $\mf{i}$ is the canonical inclusion, at the level of cohomologies $\mf{i}_*$ remains injective. We leave it to the reader to convince himself that the cohomology group sheaves of the complex of sheaves
\[\Big(\oplus_{g\in G}\Sym^{\bullet}(\mc{T}_X^g)\otimes_{\delta^{-1}\OO_{Y^g|\Delta}}\bigwedge^{\bullet}\delta^{-1}\mc{T}_{Y^g|\Delta}^*\Big)^G\] 
are actually finite dimensional which implies that $\mf{i}_*$ is an isomorphism, as desired. 
\end{proof}
\begin{proposition}
\label{scctrhcc}
There is an isomorphism of $\bb C$-vetor spaces
\[\check{\h}^2(\mc{U}, \sigma_{\geq1}\mc{C}_f^{\bullet}(\mc{D}_X\rtimes G, \mc{D}_X\rtimes G))\cong\Big(\check{\h}^2(\mc{U}, \Omega_X^{\geq1})\oplus\hspace{-1em}\bigoplus\limits_{\codim(X_i^g)=1}\hspace{-1em}\check{\h}^0(\mc{U}, j_{i*}^g\pi_{i*}^g\Omega_{T^*X_i^g}^{\bullet})\Big)^G.\]
\end{proposition}
\begin{proof} 
We abuse notation by denoting  an open cover of the orbifold $X/G$ and its preimage in the $G$-equivariant topology of $X$ by $\mc{U}$. Let for the sake of generality $\mc{L}^{\bullet}$ denote an arbitrary complex of sheaves. Let $\check{\C}^{\bullet}(\mc{U}, \mc{L}^{\bullet})$ be the \v Cech double complex thereof. Its total complex $T^{\bullet}:=\Tot^{\bullet}\big(\check{C}^{\bullet}(\mc{U}, \mc{L}^{\bullet}\big)$ has a natural decreasing filtration by the second degree
\begin{equation}
\label{filtration}
F^pT^{\bullet}:=\oplus_n\oplus_{\substack{i+j=n\\j\geq p}}\check{C}^i(\mc{U}, \mc{L}^{j})
\end{equation} 
which determines a short exact sequence
\begin{equation}
\label{filtrationses}
0\rightarrow F^{p}T^{\bullet}\rightarrow F^{p-1}T^{\bullet}\rightarrow\frac{F^{p-1}T^{\bullet}}{F^{p}T^{\bullet}}\rightarrow0
\end{equation}
for every $p$. Note the simple but important relation 
\begin{align}
\label{filtrationoftotcomplex}
F^pT^{\bullet}&=\oplus_n\oplus_{j\geq p}^n\check{\C}^{n-j}(\mc{U}, \mc{L}^{j})\nonumber\\
&=\oplus_n\oplus_{j\geq 0}\check{\C}^{n-j}\big(\mc{U}, (\sigma_{\geq p}\mc{L})^{j}\big)\nonumber\\
&=\Tot^{\bullet}\big(\check{\C}^{\bullet}(\mc{U}, \sigma_{\geq p}\mc{L}^{\bullet}\big).
\end{align}  
The above defined filtration \eqref{filtration} applied to the total complex of the \v Cech double complex of $\Big(\oplus_{i,g\in G}j_{i*}^g\pi_{i*}^g\Omega_{T^*X_i^g}^{\bullet-2l_g^i}\Big)^G$ yields in account of \eqref{filtrationoftotcomplex} the short exact sequence of complexes
\begin{align*}
&0\rightarrow\Big(\Tot^{\bullet}(\check{\C}^{\bullet}(\mc{U}, \sigma_{\geq1}\pi_{*}\Omega_{T^*X}^{\bullet}))\Big)^G\oplus\Big(\oplus_{i, g}\Tot^{\bullet}\big(\check{\C}^{\bullet}(\mc{U}, j_{i*}^g\pi_{i*}^g\Omega_{T^*X_i^g}^{\bullet-2l_g^i})\big)\Big)^G\\
&\xrightarrow{\mc{I}}\Big(\oplus_{i, g}\Tot^{\bullet}\big(\check{\C}^{\bullet}(\mc{U}, j_{i*}^g\pi_{i*}^g\Omega_{T^*X_i^g}^{\bullet-2l_g^i})\big)\Big)^G\xrightarrow{\mc{P}}\Big(\check{\C}^{\bullet}(\mc{U}, \pi_*\OO_{T^*X})\Big)^G\rightarrow0.
\end{align*}
In turn it induces a long exact sequence of \v Cech hypercohomology groups 
 \begin{align}
 \label{lesofholomorphicdiffforms}
&\cdots\rightarrow\Big(\check{\h}^1(\mc{U}, \sigma_{\geq1}\pi_*\Omega_{T^*X}^{\bullet})\Big)^G\xrightarrow{\mc{I}_*}\Big(\check{\h}^1(\mc{U}, \pi_*\Omega_{T^*X}^{\bullet})\Big)^G\xrightarrow{\mc{P}_*}\Big(\check{\h}^1(\mc{U}, \pi_*\OO_{T^*X})
\Big)^G\nonumber\\&\xrightarrow{\partial}\Big(\check{\h}^2(\mc{U}, \sigma_{\geq1}\pi_*\Omega_{T^*X}^{\bullet})\oplus_{\codim(X_i^g)=1}\check{\h}^0(\mc{U}, j_{i*}^g\pi_{i*}^g\Omega_{T^*X_i^g}^{\bullet})\Big)^G\nonumber\\
&\xrightarrow{\mc{I}_*}\Big(\check{\h}^2(\mc{U}, \pi_*\Omega_{T^*X}^{\bullet})\oplus_{\codim(X_i^g)=1}\check{\h}^0(\mc{U}, j_{i*}^g\pi_{i*}^g\Omega_{T^*X_i^g}^{\bullet})\Big)^G\xrightarrow{\mc{P}_*}\Big(\check{\h}^2(\mc{U}, \pi_*\OO_{T^*X})\Big)^G\rightarrow\cdots
 \end{align}
where $\partial$ denotes the so-called connecting morphism. A decomposition in terms of short exact sequences yields in degree $2$ the following short exact sequence
\begin{align}
\label{ses1}
&0\rightarrow\frac{\check{\h}^{1}(\mc{U}, \pi_*\OO_{T^*X})^G}{\Ima(\mc{P}_*)}\xrightarrow{\partial}\check{\h}^2(\mc{U}, \sigma_{\geq1}\pi_*\Omega_{T^*X}^{\bullet})^G\oplus(\oplus_{g, i}\check{\h}^0(\mc{U}, j_{i*}^g\pi_{i*}^g\Omega_{T^*X_i^g}^{\bullet}))^G\xrightarrow{\mc{I}_*}\ker(\mc{P}_*)\rightarrow0.
\end{align}
In a similar fashion the short exact sequence $0\rightarrow\bb C_X\xrightarrow{i}\pi_*\OO_{T^*X}\xrightarrow{p}\pi_*\OO_{T^*X}/\bb C_X\rightarrow0$ induces the long exact sequence
\begin{align*}
\cdots\rightarrow\check{\h}^1(\mc{U}, \bb C_X)&\xrightarrow{i_*}\check{\h}^1(\mc{U}, \pi_*\OO_{T^*X})\xrightarrow{p_*}\check{\h}^1(\mc{U}, \pi_*\OO_{T^*X}/\bb C_X)\xrightarrow{\partial}\check{\h}^2(\mc{U}, \bb C_X)\\&\xrightarrow{i_*}\check{\h}^2(\mc{U}, \pi_*\OO_{T^*X})\xrightarrow{p_*}\check{\h}^2(\mc{U}, \pi_*\OO_{T^*X}/\bb C_X)\rightarrow\cdots
\end{align*}
in which by abuse of notation $\partial$ again denotes the connecting morphism. It induces the short exact sequence
\begin{align}
\label{ses2}
&0\rightarrow\frac{\check{\h}^1(\mc{U}, \pi_*\OO_{T^*X})}{\Ima(i_*)}\xrightarrow{p_*}\check{\h}^1(\mc{U}, \pi_*\OO_{T^*X}/\bb C_X)\xrightarrow{\partial}\ker{i_*}\rightarrow0.
\end{align}
The pullback of the zero section $s_0: X\rightarrow T^*X$ defines a cochain map of sheaves $s_0^*: \pi_*\Omega_{T^*X}^{\bullet}\rightarrow\Omega_{X}^{\bullet}$. By definition, $s_0^*\circ\pi^*=\id$. On the other hand a holomorphic homotopy operator $K: \pi_*\Omega_{T^*X}^p\rightarrow\Omega_X^{p-1}$ can be constructed following verbatim Chapter $4$ in \cite{BT82} by means of which it can be shown that $\pi^*\circ s_0^*$ is cochain homotopic to the identity of $\pi_*\Omega_{T^*X}^{\bullet}$. Thus, $s_0^*$ is a quasi-isomorphism. Ergo, $\check{\h}^{\bullet}(\mc{U}, \pi_*\Omega_{T^*X}^{\bullet})\cong\check{\h}^{\bullet}(\mc{U}, \Omega_{X}^{\bullet})$. Consequently, 
from the fact that $\Omega_X^{\bullet}$ is a resolution of $\bb C_X$ we infer that $\Ima(\mc{P_*})\cong\Ima(i_*)$ and $\ker(\mc{P}_*)\cong\ker(i_*)$, respectively. 
The isomorphism $\pi_*\OO_{T^*X}/\bb C_X\cong\pi_*\Omega_{T^*X, \cl}^1$
yields a morphism of \v Cech cochain complexes $\kappa: \check{\C}^{\bullet}(\mc{U}, \pi_*\OO_{T^*X}/\bb C_X)\rightarrow\Tot^{\bullet}\check{\C}^{\bullet}(\mc{U}, \sigma_{\geq1}\pi_*\Omega_{T^*X}^{\bullet})[1]$, given by \[\underline{f \mod \bb C}\mapsto(-1)^nd_{dR}(\underline{f})\]
for every $\underline{f\mod\bb C}\in\check{\C}^n(\mc{U}, \pi_*\OO_{T^*X}/\bb C_X)$. Indeed, let $D$ be the differential in the total complex $\Tot^{\bullet}\check{\C}^{\bullet}(\mc{U}, \pi_*\Omega_{T^*X}^{\bullet})$. Then, 
\[\bar{D}_n:=(-1)^1D_n[1]:=\sum_{p+q=n+1}-\delta^{p,q}-(-1)^pd_{dR}^{p,q}\] 
is the differential in degree $n$ of the shifted total complex $\Tot^{\bullet}\check{\C}^{\bullet}(\mc{U}, \sigma_{\geq1}\pi_*\Omega_{T^*X}^{\bullet})[1]$. Then, for every $\underline{f\mod\bb C}\in\check{\C}^n(\mc{U}, \pi_*\OO_{T^*X}/\bb C_X)$, %
we have
\begin{align*}
\label{problem}
\kappa(\delta_n(\underline{f\mod \bb C}))&=(-1)^{n+1}d_{dR}(\delta_n(\underline{f}))\nonumber\\
&\hspace{-2em}\underbrace{=}_{\delta d_{dR}-d_{dR}\delta=0}\hspace{-2em}-\delta_n((-1)^nd_{dR}(\underline{f}))\nonumber\\
&=-(\delta_n+(-1)^{n}d_{dR})(-1)^nd_{dR}(\underline{f})\nonumber\\
&=\bar{D}_{n}\big(\kappa(\underline{f\mod \bb C})\big).
\end{align*}  
Subsequently, we fuse the short exact sequence \eqref{ses1} for the special case $G=\{\id_G\}$ with the short exact sequence \eqref{ses2} in the following commutative diagram  
\begin{equation*}
 \begin{tikzcd}[row sep=tiny]
&& \check{\h}^{1}(\mc{U}, \pi_*\OO_{T^*X}/\bb C_X) \arrow[dd, "\kappa_*"]\arrow[dr, "\partial"] \\
 0\arrow[r]&\frac{\check{\h}^1(\mc{U}, \pi_*\OO_{T^*X})}{\Ima(i_*)} \arrow[ur, "p_*"] \arrow[dr, "\partial"] & &\ker(i_*)\arrow[r]&0. \\
&&\check{\h}^2(\mc{U}, \sigma_{\geq1}\pi_*\Omega_{T^*X}^{\bullet})\arrow[ur, "\mc{I}_*"]
        \end{tikzcd}
\end{equation*}
Indeed, given an element $\underline{f}\in\check{\C}^1(\mc{U}, \pi_*\OO_{T^*X})$, for every $\underline{\alpha}^{01}\in\check{\C}^0(\mc{U}, \pi_*\Omega_{T^*X}^1)$, the cochain $\underline{\alpha}=\underline{f}+\underline{\alpha}^{01}\in\Tot^1\check{\C}^{\bullet}(\mc{U}, \pi_*\Omega_{T^*X}^{\bullet})$ satisfies $\mc{P}_1(\underline{\alpha})=\underline{f}$. Then, we have
\begin{align*}
\kappa_{*}\circ p_{*}([\underline{f}])&=\kappa_*([\underline{f \mod \bb C}])\\
&=[d_{dR}(\underline{f})]\\
&=[d_{dR}(\underline{f})+D_1(\underline{\alpha}^{01})]\\
&=[\mc{I}_2^{-1}\circ D_1(\underline{\alpha})]\\
&=\partial[\underline{f}].
\end{align*} 
Similarly, by virtue of the above and the isomorphism $\check{\h}^1(\mc{U}, \pi_*\Omega_{T^*X}^{\bullet})\cong\check{\h}^1(\mc{U}, \bb C_X)$, we have 
\begin{align*}
\mc{I}_*\circ\kappa_*([\underline{f \mod \bb C}])&=[D_1(\underline{\alpha})]\\
&=\partial[\underline{f}].
\end{align*} 
By the $5$-Lemma the linear morphism $\kappa_*$ is in fact a linear isomorphism. This coupled to the fact that $\pi_*\OO_{T^*X}/\bb C_X\rightarrow\pi_*\Omega_{T^*X}^{\geq1}[1]$ is a resolution of $\pi_*\OO_{T^*X}/\bb C_X$ implies consequently
 \begin{equation}
 \label{firsttruncatediso}
 \check{\h}^2(\mc{U}, \sigma_{\geq1}\pi_*\Omega_{T^*X}^{\bullet})\cong\check{\h}^1(\mc{U}, \pi_*\OO_{T^*X}/\bb C_X)\cong\check{\h}^2(\mc{U}, \pi_*\Omega_{T^*X}^{\geq1})\cong\check{\h}^2(\mc{U}, \Omega_{X}^{\geq1}).
 \end{equation} 
The morphisms $\mc{X}$ and $\mc{X}_{\geq1}:=\sigma_{\geq1}(\mc{X})$ induce correspondingly a quasi-isomorphism 
 \begin{equation}
 \label{totcech}
 \bar{\mc{X}}: \Tot^{\bullet}\check{\C}^{\bullet}\big(\mc{U}, \mc{C}^{\bullet}(\mc{D}_X\rtimes G, \mc{D}_X\rtimes G)\big)\rightarrow\Big(\oplus_{i, g}\Tot^{\bullet}\big(\check{\C}^{\bullet}(\mc{U}, j_{i*}^g\pi_{i*}^g\Omega_{T^*X_i^g}^{\bullet-2l_g^i})\big)\Big)^G
 \end{equation}
and a morphism\footnote{The brutally truncated map $\bar\chi_{\geq1}$ is not a quasi-isomorphism in general.}  
 \begin{equation*}
 \label{truncatedtotcech}
 \bar{\mc{X}}_{\geq1}: \Tot^{\bullet}\check{\C}^{\bullet}\big(\mc{U}, \sigma_{\geq1}\mc{C}^{\bullet}(\mc{D}_X\rtimes G, \mc{D}_X\rtimes G)\big)\rightarrow\Big(\oplus_{i, g}\Tot^{\bullet}\big(\check{\C}^{\bullet}(\mc{U}, \sigma_{\geq1}j_{i*}^g\pi_{i*}^g\Omega_{T^*X_i^g}^{\bullet-2l_g^i})\big)\Big)^G.
 \end{equation*}
 
 Let $K^{\bullet\bullet}$ 
be the \v Cech double complex of $\mc{C}_f^{\bullet}(\mc{D}_X\rtimes G, \mc{D}_X\rtimes G)$, associated to the cover $\mc{U}$, 
in which $\delta$ denotes the \v Cech differential, $d$ is the standard Hochschild differential with $\delta d-d\delta=0$. By definition, we have $\mc{C}_f^0(\mc{D}_X\rtimes G, \mc{D}_X\rtimes G)=\sHom_{\bb C}(\bb C_X, \OO_X\rtimes G)\cong\OO_X\rtimes G$. The total complex $\Tot^{\bullet}(K^{\bullet\bullet})$ has a differential $D'=\delta+(-1)^pd$ in bidegree $(p, q)$.

The natural filtration \eqref{filtration} of $\Tot^{\bullet}(K^{\bullet\bullet})$ yields in accordance with \eqref{filtrationses} and \eqref{filtrationoftotcomplex} the short exact sequence
\begin{equation}
\label{sesofdsmashg}
0\rightarrow\Tot^{\bullet}\Big(\check{\C}^{\bullet}\big(\mc{U}, \sigma_{\geq1}\mc{C}_f^{\bullet}(\mc{D}_X\rtimes G, \mc{D}_X\rtimes G)\big)\Big)\xrightarrow{\mc{I}'_*}\Tot^{\bullet}(K^{\bullet\bullet})\xrightarrow{\mc{P}'_*}\check{\C}^{\bullet}(\mc{U}, \OO_{X}\rtimes G)\rightarrow0.
\end{equation}
The canonical quasi-isomorphism $\mf{i}: \mc{C}_f^{\bullet}(\mc{D}_X\rtimes G, \mc{D}_X\rtimes G)\hookrightarrow\mc{C}^{\bullet}(\mc{D}_X\rtimes G, \mc{D}_X\rtimes G)$ from Thereom \ref{caninclqiso} induces a natural inclusion of \v Cech double complexes 
\begin{equation}
\label{inclofcdc}
\bar{\mf{i}}: \check{\C}^{\bullet}(\mc{U}, \mc{C}_f^{\bullet}(\mc{D}_X\rtimes G, \mc{D}_X\rtimes G))\hookrightarrow\check{\C}^{\bullet}(\mc{U}, \mc{C}^{\bullet}(\mc{D}_X\rtimes G, \mc{D}_X\rtimes G))
\end{equation} 
which yields an isomorphism at the level of \v Cech hypercohomologies. Combining \eqref{inclofcdc} with the quasi-isomorphism \eqref{totcech}, we obtain the  maps
\begin{align*}
&\mf{p}:=\bar{\mc{X}}\circ\bar{\mf{i}}, \quad\mf{p}_{\geq1}:=
\bar{\mc{X}}_{\geq1}\circ\bar{\mf{i}}
\end{align*}
With the help of $\mf{p}$ we arrive from  \eqref{sesofdsmashg} at the long exact sequence of \v Cech hypercohomology groups 
\begin{align*}
&\cdots\rightarrow\check{\h}^1(\mc{U}, \sigma_{\geq1}\mc{C}_f^{\bullet}(\mc{D}_X, \mc{D}_X))\xrightarrow{\mc{I}'_*}\check{\h}^1(\mc{U}, \pi_*\Omega_{T^*X}^{\bullet})^G\xrightarrow{\mc{P}'_*}\check{\h}^1(\mc{U}, \pi_*\OO_{T^*X})^G\\
&\xrightarrow{\partial'}\check{\h}^2(\mc{U}, \sigma_{\geq1}\mc{C}_f^{\bullet}(\mc{D}_X\rtimes G, \mc{D}_X\rtimes G)\xrightarrow{\mc{I}'_*}\Big(\check{\h}^2(\mc{U}, \pi_*\Omega_{T^*X}^{\bullet})\hspace{-1em}\bigoplus_{\codim(X_i^g)=1}\hspace{-1em}\check{\h}^0(\mc{U}, j_{i*}^g\pi_{i*}^g\Omega_{T^*X_i^g}^{\bullet})\Big)^G\\
&\xrightarrow{\mc{P}'_*}\check{\h}^2(\mc{U}, \pi_*\OO_{T^*X})^G\rightarrow\cdots
\end{align*}
Comparison between the long exact sequence \eqref{lesofholomorphicdiffforms} and \eqref{sesofdsmashg} delivers  $\ker(\mc{P}_*)=\ker(\mc{P}'_*)$. The last fact allows us to extract   from the above long exact sequence the following short exact sequence
\begin{align}
\label{ses3}
&0\rightarrow\frac{\Big(\check{\h}^{1}(\mc{U}, \pi_*\OO_X)\Big)^G}{\Ima(\mc{P}_*)}\xrightarrow{\partial'}\check{\h}^2(\mc{U}, \sigma_{\geq1}\mc{C}_f^{\bullet}(\mc{D}_X\rtimes G, \mc{D}_X\rtimes G))\xrightarrow{\mc{I}'_*}\ker(\mc{P}_*)\rightarrow0.
\end{align}
By virtue of these morphism we can now merge the short exact sequences \eqref{ses1} and \eqref{ses3} in the ensuing diagram  
 \begin{equation}
 \label{fundamentaldiagram}
 \begin{tikzcd}[row sep=tiny]
&& M\arrow[dd, "\mf{p}_{\geq1*}"]\arrow[dr, "\mc{I}'_*"] \\
 0\arrow[r]& \frac{\check{\h}^1(\mc{U}, \pi_*\OO_{T^*X})^G}{\Ima(\mc{P}_*)} \arrow[ur, "\partial'"] \arrow[dr, "\partial\circ\mf{p}"] & &\ker(\mc{P}_*)\arrow[r]&0.           \\
&&N\arrow[ur, "\mc{I}_*"]
        \end{tikzcd}
\end{equation} 
where 
\begin{align*}
M&:=\check{\h}^{2}(\mc{U}, \sigma_{\geq1}\mc{C}_f^{\bullet}(\mc{D}_X\rtimes G, \mc{D}_X\rtimes G)),\quad N:=\big(\check{\h}^2(\mc{U}, \sigma_{\geq1}\pi_*\Omega_{T^*X}^{\bullet})\hspace{-0.4em}\bigoplus_{\codim(X_i^g)=1}\hspace{-0.4em}\check{\h}^0(\mc{U}, j_{i*}^g\pi_{i*}^g\Omega_{T^*X_i^g}^{\bullet})\big)^G.
\end{align*}
In order to show that diagram \eqref{fundamentaldiagram} commutes, consider the diagram
\begin{equation}
\small\small\begin{tikzcd}
\label{mapofexactsequences}
0\arrow[r]&\Tot^{\bullet}\check{\C}^{\bullet}(\mc{U},  \sigma_{\geq1}\mc{C}_f^{\bullet}(\mc{D}_X, \mc{D}_X))\arrow[d, "\mf{p}_{\geq1}"]\arrow[r, "\mc{I}'"]&\Tot^{\bullet}(K^{\bullet\bullet})\arrow[d, "\mf{p}"]\arrow[r, "\mc{P}'"]&K^{\bullet, 0}\arrow[d, "\mf{p}"]\arrow[r]&0\\
0\arrow[r]&\Tot^{\bullet}\check{\C}^{\bullet}\big(\mc{U}, \big(\oplus_{i,g\in G}\sigma_{\geq1}j_{i*}^g\pi_{i*}^g\Omega_{T^*X_i^g}^{\bullet-2l_g^i}\big)^G\big)\arrow[r, "\mc{I}"]&\Tot^{\bullet}L^{\bullet\bullet}\arrow[r, "\mc{P}"]&\check{\C}^{\bullet}(\mc{U}, \pi_*\OO_{T^*X})^G\arrow[r]&0.
\end{tikzcd}
\end{equation}
with $L^{\bullet\bullet}:=\check{\C}^{\bullet}\big(\mc{U}, (\oplus_{i,g\in G}j_{i*}^g\pi_{i*}^g\Omega_{T^*X_i^g}^{\bullet-2l_g^i})^G\big)$ where both horizontal diagrams are short exact sequences. The left-hand square of the diagram commutes because $\mc{I}$ and $\mc{I}'$ are the canonical inclusions. From this we immediately infer that the right-hand square of diagram \eqref{fundamentaldiagram} is commutative, too. On the other hand the right-hand square in diagram \eqref{mapofexactsequences} commutes, since for any $\underline{\beta}\in\Tot^{\bullet}(K^{\bullet\bullet})$, we have
\begin{align}
\label{rhsquare}
\mf{p}\circ\mc{P}'(\underline{\beta})&=\mf{p}(\underline{\beta} \mod \Ima(\mc{I}'))\nonumber\\
&=\mf{p}(\underline{\beta}) \mod \Ima(\mf{p}\circ\mc{I}')\nonumber\\
&=\mf{p}(\underline{\beta}) \mod \Ima(\mc{I}\circ\mf{p}_{\geq1})\nonumber\\
&=\mc{P}\circ\mf{p}(\underline{\beta}).
\end{align}
where in the third line we used the commutativity of the left square in Diagram \eqref{mapofexactsequences}. Finally, let the cohomology class $[\underline{f}]$ be an arbitrary representative in $\frac{\check{\h}^1(\mc{U}, \pi_*\OO_{T^*X})^G}{\Ima(\mc{P}_*)}$. There is an element $\underline{\alpha}\in\Tot^{\bullet}(K^{\bullet\bullet})$ such that $\mc{P}'_1(\underline{\alpha})=\underline{f}$. Then, invoking the definition of the connecting morphism $\partial'$, we obtain 
\begin{align*}
\mf{p}_{\geq1*}\circ\partial'([\underline{f}])&=\mf{p}_{\geq1*}\circ[\mc{I}_2'^{-1}D'_1(\underline{\alpha})]\\
&=[\mf{p}_{\geq1}\circ D'_1(\underline{\alpha})]\\
&=[D_1(\mf{p}_{\geq1}(\underline{\alpha}))]\\
&=[\mc{I}_2^{-1}\circ D_1(\mf{p}_{\geq1}(\underline{\alpha}))]\\
&=\partial\circ\mf{p}([\underline{f}])
\end{align*}
where in the second line we used that $\mc{I}'$ is the canonical inclusion map and  $D'_1(\underline{\alpha})\in\Ima(\mc{I}'_2)$\footnote{Since $f$ is by definition a cocycle, $D_1(\underline{\alpha})\in\ker(\mc{P}'_2)=\Ima(\mc{I}'_2)$.}, in the third the fact that $\mf{p}_{\geq1}$ is a map of cocoain complexes and in the last line we used Equality \eqref{rhsquare}. We conclude that the left-hand side of diagram \eqref{fundamentaldiagram} commutes,  whence the whole diagram is commutative. Then by the $5$-Lemma, the map $\mf{p}_{\geq1*}$ is a linear isomorphism. This combined with Isomorphism \eqref{firsttruncatediso} proves the claim. 
\end{proof}
Theorem \ref{filtdeformationspace} coupled with Proposition \ref{scctrhcc} yields the following important isomorphism.
\begin{corollary}
\label{spaceoffiltdefs}
There is an isomorphism of $\bb C$-vector spaces 
\[\sDef(\mc{D}_X\rtimes G)_f\cong\bb H^{2}(X, \Omega_X^{\geq1})^G\oplus\big(\bigoplus\limits_{\codim_{\bb C}(X_i^g)=1}\hspace{-1em}\h^0(X_i^g, \bb C)\big)^G.\] 
\end{corollary} 
\begin{proof}
The only non-trivial part in the statement is the identification of 
$\check{\h}^0(X, j_{i*}^g\pi_{i*}^g\Omega_{T^*X_i^g}^{\bullet})$ with $\h^0(X_i^g, \bb C)$. We demonstrate that now. Let $\mc{U}$ denote a $G$-invariant open cover of $X$ and let the notation $X_i^g\cap\mc{U}$ on $X_i^g$ denote the induced open cover on $X_i^g$. With this, we have 
\begin{align*}
\check{\h}^{\bullet}(X, j_{\langle g\rangle*}\pi_{i*}^g\Omega_{T^*X_i^g}^{\bullet})&=\colim_{\mc{U}}\check{\h}^{\bullet}(\mc{U}, j_{\langle g\rangle*}\pi_{i*}^g\Omega_{T^*X_i^g}^{\bullet})\\
&=\colim_{\mc{U}}\check{\h}^{\bullet}(X_i^g\cap\mc{U}, \pi_{i*}^g\Omega_{T^*X_i^g}^{\bullet})\\
&\cong\colim_{\mc{U}}\check{\h}^{\bullet}(X_i^g\cap\mc{U}, \mc{A}_{X_i^g, \bb C}^{\bullet})\\
&\cong\bb{H}^{\bullet}(X_i^g,  \mc{A}_{X_i^g, \bb C}^{\bullet})\\
&\cong\h^{\bullet}(X_i^g, \bb C)
\end{align*}
where in the third isomorphism we applied Poincare's Lemma and the fact that $ \mc{A}_{X_i^g}^{\bullet}$ and $ \Omega_{X_i^g}^{\bullet}$ are both resolutions of $\bb C$. 
\end{proof}
We arrive at the main theorem of the paper. 
\begin{theorem}
\label{univfiltformdefthm}
Let $X$ be a smooth algebraic variety or a smooth analytic variety equipped with a finite subgroup $G\subset\Aut(X)$ acting faithfully on $X$. The sheaf of twisted Cherednik algebras $\mc{H}_{1, c, \psi, X, G}$ on the quotient orbifold $X/G$ with formal $c$ and $\psi$ is a universal formal filtered deformation of the sheaf of filtered skew-group algebra $\mc{D}_X\rtimes G$. 
\end{theorem}
\begin{proof}
The claim follows from the fact that the dimension of the parameter space $\{(\psi, c)\}$ is the same as the dimension of $\sDef(\mc{D}_X\rtimes G)_f$. 
\end{proof}
\section*{Acknowledgement}
I would like to express my gratitude to the Department of Mathematics at MIT for the excellent research conditions which enabled this work at first place. I would like to extend my sincere gratitude to Prof. Pavel Etingof for many valuable comments and suggestions which he generously provided in the course of my work on this note and which substantially improved its quality. In particular, I thank him for giving me the idea to use filtration preserving Hochschild cochains- a trick which "unlocked" the central problem. I also thank Prof. Giovanni Felder, Prof. Ajay Ramadoss,  Prof. Valery Lunts and Dr. Konstantin Jacob for a useful exchange and valuable comments.  This work is financially supported by a \emph {Swiss National Science Foundation Early Postdoc.Mobility Fellowship} number 188014.     
\printbibliography
\vspace{5ex}
\textsc{Department of Mathematics, Massachusetts Institute of Technology, 77 Massachusetts Ave.,\\ Cambridge, MA 02139, USA}\\
\\
\textsc{E-mail address}: \textit{avitanov@mit.edu}

\end{document}